\documentclass[12pt,final]{amsart}
\usepackage{xspace,amssymb,amsfonts,epsfig,euscript,eufrak,enumerate,amsmath,nicefrac}
\usepackage{tikz,etex,bbm,xspace,afterpage}
 \usepackage{epstopdf,ifpdf,todonotes}
\usepackage{pxfonts,multicol}
\usepackage{stackengine}
\def\stacktype{L}
\stackMath

\usepackage{xr,xr-hyper}
\usepackage[notref,notcite]{showkeys}
\usepackage[margin=2.9cm,footskip=24pt,headheight=20pt]{geometry}
\usepackage{fancyhdr,mathrsfs,hyperref}
\usepackage[polutonikogreek,english]{babel}

\newcommand{\tU}{\mathcal{U}}

\ifpdf
  \usepackage{pdfsync}
\fi

\begin{document}
\newtheoremstyle{all}
  {11pt}
  {11pt}
  {\slshape}
  {}
  {\bfseries}
  {}
  {.5em}
  {}

\theoremstyle{all}
\newtheorem{theorem}{Theorem}[section]
\newtheorem{proposition}[theorem]{Proposition}
\newtheorem{corollary}[theorem]{Corollary}
\newtheorem{lemma}[theorem]{Lemma}
\newtheorem{definition}[theorem]{Definition}
\newtheorem{ques}[theorem]{Question}
\newtheorem{conj}[theorem]{Conjecture}
\newtheorem{rem}[theorem]{Remark}

\newcommand{\nc}{\newcommand}
\newcommand{\renc}{\renewcommand}
  \nc{\kac}{\kappa^C}
\nc{\alg}{T}
\nc{\supp}{\operatorname{supp}}
\nc{\msupp}{\operatorname{\mu supp}}

\nc{\Lco}{L_{\la}}
\nc{\qD}{q^{\nicefrac 1D}}
\nc{\ocL}{M_{\la}}
\nc{\excise}[1]{}
\nc{\Dbe}{D^{\uparrow}}
\nc{\Dfg}{D^{\mathsf{fg}}}

\nc{\tr}{\operatorname{tr}}
\newcommand{\Mirkovic}{Mirkovi\'c\xspace}
\nc{\tla}{\mathsf{t}_\la}
\nc{\llrr}{\langle\la,\rho\rangle}
\nc{\lllr}{\langle\la,\la\rangle}
\nc{\K}{\mathbbm{k}}
\nc{\Stosic}{Sto{\v{s}}i{\'c}\xspace}
\nc{\cd}{\mathcal{D}}
\nc{\cT}{\mathcal{T}}
\nc{\vd}{\mathbb{D}}
\nc{\R}{\mathbb{R}}
\renc{\wr}{\operatorname{wr}}
  \nc{\Lam}[3]{\La^{#1}_{#2,#3}}
  \nc{\Lab}[2]{\La^{#1}_{#2}}
  \nc{\Lamvwy}{\Lam\Bv\Bw\By}
  \nc{\Labwv}{\Lab\Bw\Bv}
  \nc{\nak}[3]{\mathcal{N}(#1,#2,#3)}
  \nc{\hw}{highest weight\xspace}
  \nc{\al}{\alpha}
  \nc{\be}{\beta}
  \nc{\bM}{\mathbf{m}}
  \nc{\bkh}{\backslash}
  \nc{\Bi}{\mathbf{i}}
  \nc{\Bj}{\mathbf{j}}
\nc{\bd}{\mathbf{d}}
\nc{\D}{\mathcal{D}}
\nc{\mmod}{\operatorname{-mod}}  
\nc{\immod}{\operatorname{-\widehat{mod}}}  
\newcommand{\red}{\mathfrak{r}}

\nc{\RAA}{R^\A_A}
  \nc{\Bv}{\mathbf{v}}
  \nc{\Bw}{\mathbf{w}}
\nc{\Id}{\operatorname{Id}}
  \nc{\By}{\mathbf{y}}
\nc{\eE}{\EuScript{E}}
  \nc{\Bz}{\mathbf{z}}
  \nc{\coker}{\mathrm{coker}\,}
  \nc{\C}{\mathbb{C}}
  \nc{\ch}{\mathrm{ch}}
  \nc{\de}{\delta}
  \nc{\ep}{\epsilon}
  \nc{\Rep}[2]{\mathsf{Rep}_{#1}^{#2}}
  \nc{\Ev}[2]{E_{#1}^{#2}}
  \nc{\fr}[1]{\mathfrak{#1}}
  \nc{\fp}{\fr p}
  \nc{\fq}{\fr q}
  \nc{\fl}{\fr l}
  \nc{\fgl}{\fr{gl}}
\nc{\rad}{\operatorname{rad}}
\nc{\ind}{\operatorname{ind}}
  \nc{\GL}{\mathrm{GL}}
  \nc{\Hom}{\mathrm{Hom}}
  \nc{\im}{\mathrm{im}\,}
  \nc{\La}{\Lambda}
  \nc{\la}{\lambda}
  \nc{\mult}{b^{\mu}_{\la_0}\!}
  \nc{\mc}[1]{\mathcal{#1}}
  \nc{\om}{\omega}
\nc{\gl}{\mathfrak{gl}}
  \nc{\cF}{\mathcal{F}}
 \nc{\cC}{\mathcal{C}}
  \nc{\Mor}{\mathsf{Mor}}
  \nc{\Ob}{\mathsf{Ob}}
  \nc{\Vect}{\mathsf{Vect}}
\nc{\gVect}{\mathsf{gVect}}
  \nc{\modu}{\mathsf{-mod}}
\nc{\pmodu}{\mathsf{-pmod}}
  \nc{\qvw}[1]{\La(#1 \Bv,\Bw)}
  \nc{\van}[1]{\nu_{#1}}
  \nc{\Rperp}{R^\vee(X_0)^{\perp}}
  \nc{\si}{\sigma}
  \nc{\croot}[1]{\al^\vee_{#1}}
\nc{\di}{\mathbf{d}}
  \nc{\SL}[1]{\mathrm{SL}_{#1}}
  \nc{\Th}{\theta}
  \nc{\vp}{\varphi}
  \nc{\wt}{\mathrm{wt}}
\nc{\te}{\tilde{e}}
\nc{\tf}{\tilde{f}}
\nc{\hwo}{\mathbb{V}}
\nc{\soc}{\operatorname{soc}}
\nc{\cosoc}{\operatorname{cosoc}}
 \nc{\Q}{\mathbb{Q}}

  \nc{\Z}{\mathbb{Z}}
  \nc{\Znn}{\Z_{\geq 0}}
  \nc{\ver}{\EuScript{V}}
  \nc{\Res}[2]{\operatorname{Res}^{#1}_{#2}}
  \nc{\edge}{\EuScript{E}}
  \nc{\Spec}{\mathrm{Spec}}
  \nc{\tie}{\EuScript{T}}
  \nc{\ml}[1]{\mathbb{D}^{#1}}
  \nc{\fQ}{\mathfrak{Q}}
        \nc{\fg}{\mathfrak{g}}
  \nc{\Uq}{U_q(\fg)}
        \nc{\bom}{\boldsymbol{\omega}}
\nc{\bla}{{\underline{\boldsymbol{\la}}}}
\nc{\bmu}{{\underline{\boldsymbol{\mu}}}}
\nc{\bal}{{\boldsymbol{\al}}}
\nc{\bet}{{\boldsymbol{\eta}}}
\nc{\rola}{X}
\nc{\wela}{Y}
\nc{\fM}{\mathfrak{M}}
\nc{\fN}{\mathfrak{N}}
\nc{\fX}{\mathfrak{X}}
\nc{\fH}{\mathfrak{H}}
\nc{\fE}{\mathfrak{E}}
\nc{\fF}{\mathfrak{F}}
\nc{\fI}{\mathfrak{I}}
\nc{\qui}[2]{\fM_{#1}^{#2}}
\nc{\cL}{\mathcal{L}}
\nc{\cM}{\mathcal{M}}
\nc{\ca}[2]{\fQ_{#1}^{#2}}
\nc{\cat}{\mathcal{V}}
\nc{\cata}{\mathfrak{V}}
\nc{\catf}{\mathscr{V}}
\nc{\hl}{\mathcal{X}}
\nc{\pil}{{\boldsymbol{\pi}}^L}
\nc{\pir}{{\boldsymbol{\pi}}^R}
\nc{\cO}{\mathcal{O}}
\nc{\Ko}{\text{\Denarius}}
\nc{\Ei}{\fE_i}
\nc{\Fi}{\fF_i}
\nc{\fil}{\mathcal{H}}
\nc{\brr}[2]{\beta^R_{#1,#2}}
\nc{\brl}[2]{\beta^L_{#1,#2}}
\nc{\so}[2]{\EuScript{Q}^{#1}_{#2}}
\nc{\EW}{\mathbf{W}}
\nc{\rma}[2]{\mathbf{R}_{#1,#2}}
\nc{\Dif}{\EuScript{D}}
\nc{\MDif}{\EuScript{E}}
\renc{\mod}{\mathsf{mod}}
\nc{\modg}{\mathsf{mod}^g}
\nc{\fmod}{\mathsf{mod}^{fd}}
\nc{\id}{\mathbbm{1}}
\nc{\DR}{\mathbf{DR}}
\nc{\End}{\operatorname{End}}
\nc{\Fun}{\operatorname{Fun}}
\nc{\Ext}{\operatorname{Ext}}
\nc{\tw}{\tau}
\nc{\A}{\EuScript{A}}
\nc{\Loc}{\mathsf{Loc}}
\nc{\eF}{\EuScript{F}}
\nc{\LAA}{\Loc^{\A}_{A}}
\nc{\perv}{\mathsf{Perv}}
\nc{\gfq}[2]{B_{#1}^{#2}}
\nc{\qgf}[1]{A_{#1}}
\nc{\qgr}{\qgf\rho}
\nc{\tqgf}{\tilde A}
\nc{\Tr}{\operatorname{Tr}}
\nc{\Tor}{\operatorname{Tor}}
\nc{\cQ}{\mathcal{Q}}
\nc{\st}[1]{\Delta(#1)}
\nc{\cst}[1]{\nabla(#1)}
\nc{\ei}{\mathbf{e}_i}
\nc{\Be}{\mathbf{e}}
\nc{\Hck}{\mathfrak{H}}
\renc{\P}{\mathbb{P}}
\nc{\bbB}{\mathbb{B}}

\nc{\cI}{\mathcal{I}}
\nc{\cG}{\mathcal{G}}
\nc{\cH}{\mathcal{H}}
\nc{\coe}{\mathfrak{K}}
\nc{\pr}{\operatorname{pr}}
\nc{\bra}{\mathfrak{B}}
\nc{\rcl}{\rho^\vee(\la)}

\nc{\dU}{{\stackon[8pt]{\tU}{\cdot}}}
\newcommand{\arxiv}[1]{\href{http://arxiv.org/abs/#1}{\tt arXiv:\nolinkurl{#1}}}
\nc{\RHom}{\mathrm{RHom}}
\nc{\tcO}{\tilde{\cO}}
\nc{\Yon}{\mathscr{Y}}
\nc{\sI}{{\mathsf{I}}}

\setcounter{tocdepth}{1}
\numberwithin{equation}{section}
\newcounter{subeqn}
\renewcommand{\thesubeqn}{\theequation\alph{subeqn}}
\newcommand{\subeqn}{%
  \refstepcounter{subeqn}
  \tag{\thesubeqn}
}\makeatletter
\@addtoreset{subeqn}{equation}
\newcommand{\newseq}{%
  \refstepcounter{equation}
}
\excise{
\newenvironment{block}
\newenvironment{frame}
\newenvironment{tikzpicture}
\newenvironment{equation*}
}

\baselineskip=1.1\baselineskip

 \usetikzlibrary{decorations.pathreplacing,backgrounds,decorations.markings}
\tikzset{wei/.style={draw=red,double=red!40!white,double distance=1.5pt,thin}}
\tikzset{awei/.style={draw=blue,double=blue!40!white,double distance=1.5pt,thin}}
\tikzset{bdot/.style={fill,circle,color=blue,inner sep=3pt,outer
    sep=0}}
\tikzset{dir/.style={thick,postaction={decorate,decoration={markings,
    mark=at position .8 with {\arrow[scale=1.3]{>}}}}}}
\tikzset{rdir/.style={postaction={decorate,decoration={markings,
    mark=at position .8 with {\arrow[scale=1.3]{>}}}}}}
\tikzset{edir/.style={postaction={decorate,decoration={markings,
    mark=at position .2 with {\arrow[scale=1.3]{<}}}}}}
\begin{center}
\noindent {\large  \bf A categorical action on quantized quiver varieties}
\medskip

\noindent {\sc Ben Webster}\footnote{Supported by the NSF under Grant DMS-1151473 and  by the NSA under Grant H98230-10-1-0199.}\\  
Department of Mathematics\\ Northeastern University\\
Boston, MA\\
Email: {\tt b.webster@neu.edu}
\end{center}
\bigskip
{\small
\begin{quote}
\noindent {\em Abstract.}
In this paper, we describe a categorical action of any symmetric Kac-Moody
algebra on a category of quantized coherent sheaves on
Nakajima quiver varieties.  By ``quantized coherent sheaves,'' we mean a category of sheaves of
modules over a deformation quantization of the natural symplectic
structure on quiver varieties.  This
action is a direct categorification of the geometric construction of universal
enveloping algebras by Nakajima.
\end{quote}
}

\vspace{1cm}

\renc{\thetheorem}{\Alph{theorem}}

\tableofcontents

Let $\fg$ be an arbitrary Kac-Moody algebra with symmetric Cartan
matrix, and $\Gamma$ its associated Dynkin graph.  Nakajima showed
that there exists a remarkable connection between the algebra $U(\fg)$
and certain varieties, called {\bf quiver varieties}, constructed
directly from the graph $\Gamma$.  This construction takes the form of a
map from $U(\fg)$ to the Borel-Moore homology of a quiver analogue of the
Steinberg variety \cite{Nak98}.

Both the source and target of this map have natural categorifications:
\begin{itemize}
\item the algebra $U(\fg)$ is categorified by a 2-category
  $\tU$.  Actually several variations on the theme of this category
  have been introduced by Rouquier \cite{Rou2KM}, Khovanov-Lauda
  \cite{KLIII}, and Cautis-Lauda \cite{CaLa}; we will use the
  formalism of the last of these.  A 2-functor from this category into
  another 2-category is called a {\bf categorical action} of $\fg$ in
  this 2-category.
\item the Borel-Moore homology of the ``Steinberg'' of a symplectic
  resolution $\fM$ (such as a quiver variety) is categorified by a certain category
  of sheaves on $\fM\times \fM$.  The structure sheaf of $\fM$
  possesses a quantization, in the sense of \cite{BK04a}, and the
  category of interest to us is that of bimodules over this quantization
  which satisfy a ``Harish-Chandra'' property, as described by Braden,
  Proudfoot and the author \cite[\S 6.1-2]{BLPWquant}.  Viewed correctly,
  these bimodules on the quiver varieties associated to a single
  highest weight $\la$ can be organized into a 2-category, which we denote $\EuScript{Q}^\la$.
\end{itemize}
Thus, this previous work suggests how to categorify Nakajima's map:
\begin{theorem}\label{th:main}
For each highest weight
  $\la$, there is a categorical representation of $\fg$ in the
  2-category $\EuScript{Q}^\la$; taking ``characteristic cycles'' of these
  bimodules recovers the geometric construction of
  $\dot{U}$ by Nakajima. 
\end{theorem}
Furthermore, the form of this functor is strongly suggested by
Nakajima's work; his map is defined by sending the Chevalley
generators of $U(\fg)$ to particular correspondences, called {\bf Hecke
  correspondences}, which have natural moduli-theoretic significance.
We ``upgrade'' these correspondences to modules over deformation
quantizations, and show that these satisfy the
categorical analogues of the Chevalley presentation.
For the experts, we should note that this will not work with arbitrary
quantizations.  The quantizations we wish to consider are classified
up to isomorphism by classes in $H^2(\fM;\C)$ called {\bf periods}.
The correspondences can only be quantized when the period satisfies an
integrality condition, as we'll discuss in much more detail in
Section \ref{sec:non-integral-case}.

We regard this theorem as very strong evidence of the naturality of
the notion of a categorical $\fg$-action currently circulating in the
literature.  While defined diagrammatically in a way that might
outwardly seem arbitrary, in fact, its relations are hard-coded in the
geometry of quiver varieties.

This action of a 2-category is also quite useful in understanding
categories of sheaves on quiver varieties.  In particular, we'll use
it to understand the category of {\bf core modules} for certain
``integral'' quantizations. These are closely
related to the category of finite dimensional modules over a global
quantization of the quiver variety. Work
of Bezrukavnikov and Losev \cite{BLet} following
up on this paper has described this category for more general
quantizations, resolving a conjecture of Etingof on the structure of
finite dimensional modules over a symplectic reflection algebra.\medskip

This theorem fits into a context of older results.  
Very close analogues of the functors that appear in this
representation have already been constructed in work of Zheng
\cite{Zheng2008} and Li \cite{Li10,Li10b}. However, these authors work
in a slightly different context, which is based on constructible sheaves rather
than deformation quantizations.  The Riemann-Hilbert correspondence
has already established a tie between constructible sheaves on a
space $X$, and certain modules over a deformation quantization of
$T^*X$: the differential operators $\D_{X}$ on $X$.  From this perspective,  if
there were a space $Y$ of which a given Nakajima quiver variety were the
cotangent bundle (there almost never is) then 
sheaves of modules over the quantized structure sheaf could be thought
of as a replacement for the category of D-modules on the hypothetical space $Y$. As pointed out by Zheng \cite[\S
2.2]{Zheng2008}, his work was in a sense intended to understand
constructible sheaves with the same philosophy.  

Rouquier \cite[5.10]{RouQH} showed that Zheng's action can
be strengthened to an action of two 2-category $\tU$;
while it is not obvious that Rouquier's category is the same as that
from \cite{CaLa}, this was later proven by Brundan \cite{Brundandef}.  Rouquier's result is extremely close to the first clause
of Theorem \ref{th:main}, but a host of annoying details rise up if
one tries to derive one from the other: the result \cite[5.10]{RouQH} only establishes that the
functors induce an action on a subcategory of Zheng's category, though
this proof could likely be extended; all the above work is on
$\mathbb{Q}_\ell$-sheaves on a variety over finite fields rather than
over $\C$, etc.  None of these issues are insuperable, but we felt the
reader would be better served by an exposition which is more native
to the world of deformation quantizations.

We are also motivated by analogous results that have appeared in the
literature on coherent sheaves, for example in the work of Cautis,
Licata and Kamnitzer \cite{CK3,CKLHowe,CKL2,CKLquiver,CKL1}.  Amongst
other things, these results show that the categories of coherent
sheaves on quiver varieties carry a version of a categorical action.  In
particular, these results have lead to interesting equivalences
between derived categories of coherent sheaves.  From our perspective,
the action on sheaves over deformation quantizations is easier to work
with, since one can use topological methods for D-modules, and seems
to be the more basic object.  In forthcoming work, Cautis, Dodd and
Kamnitzer \cite{CDK} will make between classical and quantum
situations precise, showing that the action on coherent sheaves of
quiver varieties in \cite{CKLquiver} is a classical limit of the
action presented here.

More generally, this action is but one aspect of close ties between
the geometry of quiver varieties and the theory of categorical Lie
algebra actions.  It builds on work of Rouquier, Varagnolo and
Vasserot \cite{RouQH,VV} and is expanded further in further
work of the author \cite{Webqui}, which relates other categories
of modules over these deformation quantizations to known categorical $\mathfrak{g}$-actions.\medskip

Another perspective on these deformation quantizations is that they provide
a replacement for the Fukaya category of a complex symplectic
variety.  Such a connection is suggested by Kapustin and Witten
\cite[\S 11]{KW07} from a physical perspective, and the
work of Nadler and Zaslow \cite{NZ} relating constructible sheaves and the
Fukaya category of a cotangent bundle is also quite suggestive along
these lines.  In
particular, it would be very interesting to find a categorical Lie
algebra action in the 2-category of Lagrangian correspondences
constructed by Wehrheim and Woodward \cite{WeWo}.
Hopefully, instead of finding modules supported on the Hecke correspondences, one
would simply consider them as objects in the Fukaya category.

Our main technical tool is a theorem of Rouquier \cite[4.13]{RouQH}  which greatly reduces the number of relations
which need to be checked in order to confirm that a candidate is a
categorical action.   This result is quite similar to earlier works of
Chuang and Rouquier (\cite[5.27]{CR04} \& \cite[5.27]{Rou2KM}) and Cautis and Lauda
\cite[Th. 1.1]{CaLa}, which
likewise reduce the number of calculations needed, but which require
stronger hypotheses. In particular, we can rely on calculations of Varagnolo
and Vasserot from \cite{VV} for the most important check of relations
between 2-morphisms; the other conditions either follow from
general principles or are close analogues of results proven by Zheng
and Li, with proofs that can be adapted.

\subsection*{Acknowledgements}

This paper owes a great debt to
Yiqiang Li; his work was an important inspiration, and he very
helpfully pointed out a serious mistake in a draft version.  I also
want to thank Nick Proudfoot, Tony Licata and
Tom Braden; I depended very much on previous work and conversations
with them to be able to write this paper.  I thank
Sabin Cautis and Aaron Lauda for sharing an early version of their
paper with me.  I also appreciate very stimulating conversations with
Catharina Stroppel, Ivan Losev and Peter Tingley.

\subsection*{Notation}

We let $\Gamma$ be an oriented graph and $\fg$ the associated Kac-Moody algebra.  Consider the
weight lattice $\wela(\fg)$ and root lattice $\rola(\fg)$, and the
simple roots $\al_i$ and coroots $\al_i^\vee$.  Let
$c_{ij}=\al_j^{\vee}(\al_i)$ be the entries of the Cartan matrix.  

Choose an orientation $\Omega$ on $\Gamma$, let $\epsilon_{ij}$ denote the number of edges
oriented from $i$ to $j$, and
fix $$Q_{ij}(u,v)=(-1)^{\ep_{ij}}(u-v)^{c_{ij}}.$$

We let $U_q(\fg)$ denote the deformed universal enveloping algebra of
$\fg$; that is, the 
associative $\C(q)$-algebra given by generators $E_i$, $F_i$, $K_{\xi}$ for $i$ and $\xi \in \wela(\fg)$, subject to the relations:
\begin{center}
\begin{enumerate}[i)]
 \item $K_0=1$, $K_{\xi}K_{\xi'}=K_{\xi+\xi'}$ for all $\xi,\xi' \in \wela(\fg)$,
 \item $K_{\xi}E_i = q^{\al_i^{\vee}(\xi)}E_iK_{\xi}$ for all $\xi \in
 \wela(\fg)$,
 \item $K_{\xi}F_i = q^{ \al_i^{\vee}(\xi)}F_iK_{\xi}$ for all $\xi \in
 \wela(\fg)$,
 \item $E_iF_j - F_jE_i = \delta_{ij}
 \frac{\tilde{K}_i-\tilde{K}_{-i}}{q-q^{-1}}$, where
 $\tilde{K}_{\pm i}=K_{\pm d_i \al_i}$,
 \item For all $i\neq j$ $$\sum_{a+b=-c_{ij}+1}(-1)^{a} E_i^{(a)}E_jE_i^{(b)} = 0
 \qquad {\rm and} \qquad
 \sum_{a+b=-c_{ij} +1}(-1)^{a} F_i^{(a)}F_jF_i^{(b)} = 0 .$$
\end{enumerate} \end{center}

\renc{\thetheorem}{\arabic{section}.\arabic{theorem}}

\section{The 2-category $\mathcal{U}$}
\label{sec:2-category-cu}

Our primary object of study is a 2-category categorifying the
universal enveloping algebra; versions of this category have been
considered by Rouquier \cite{Rou2KM}, Khovanov and Lauda \cite{KLIII}
and Cautis and Lauda \cite{CaLa}.  Since recent work of Brundan
\cite{Brundandef} has shown that the different definitions given in
these papers are equivalent, we will work with the definition given in
\cite{Rou2KM}.  For simplicity of notation, if $u_1\dots u_n$ is the
composition of $n$ 1-morphisms in a 2-category, we let  $x^{(\ell)}$
for $x\colon u_\ell\to u_\ell$ a 2-morphism horizontal composition 
$1_{u_1}\otimes 1_{u_2}\otimes \cdots\otimes  1_{u_{\ell-1}}\otimes
x\otimes 1_{u_{\ell+1}}\otimes \cdots\otimes  1_{u_n}$, and similarly
with 
$x^{(\ell,\ell+1)}$ for  $x\colon u_\ell u_{\ell+1}\to u_\ell u_{\ell+1}$.
\begin{definition}
  $\tU$ is the 2-category with:
  \begin{itemize}
  \item objects given by the weight lattice;
\item 1-morphisms freely generated under composition and direct sum by adjoint 1-morphisms $\eF_i$ and their
    right adjoints $\eE_i$ via the (co)unit 
\[\iota\colon \id_\la \to \eE_i\eF_i \id_\la\qquad \epsilon\colon \eF_i\eE_i \id_\la\to \id_\la\qquad\text{ for }i\in
    \Gamma;\]
\item 2-morphisms
    \[y_i\colon \eF_i\to \eF_i \qquad \psi_{ij}\colon \eF_i \eF_j\to
    \eF_j\eF_i\]
    \[ \xi_{i,j,\la}\colon  \eE_i \eF_j
\id_\la\oplus\id_\la^{\oplus
      \delta_{ij}\max(0,-\la^i)} \to \eF_j\eE_i \id_\la \oplus \id_\la^{\oplus \delta_{ij}\max(0,\la^i)}.\]
  \end{itemize}
These 2-morphisms are subject to the relations:
\begin{align*}
\psi_{ij} y_i^{(1)}
&= 
\begin{cases}
  y_i^{(2)}\psi_{ij}+1 &\hbox{if $i=j$},\\
   y_i^{(2)}\psi_{ij}\hspace{48mm}&\hbox{if $i\neq j$};
\end{cases}\\
 y_j^{(1)}\psi_{ij} 
&= 
\begin{cases}
 \psi_{ij} y_j^{(2)}+1 &\hbox{if $i=j$},\\
  \psi_{ij} y_j^{(2)}\hspace{48mm}&\hbox{if $i\neq j$};
\end{cases}\\
\psi_{ji}\psi_{ij} &=
\begin{cases}
0 \hspace{61mm}&\text{if $i=j$},\\
Q_{ij}(y_i^{(1)},y_j^{(2)})&\text{$i\neq j$};
\end{cases}
\end{align*}\begin{align*}
\psi_{jk}^{(1,2)}\psi_{ik}^{(2,3)} \psi_{ij}^{(1,2)} 
&=
\begin{cases}\displaystyle
  \psi_{ij}^{(2,3)} \psi_{ik}^{(1,2)}\psi_{jk}^{(2,3)} 
  +\frac{Q_{ij}(y_k^{(1)}, y_j^{(2)})-Q_{ij}(y_k^{(3)}, y_j^{(2)})}{y_i^{(1)}-y_k^{(3)}}&\text{if $i=k$},\\
\psi_{ij}^{(2,3)} \psi_{ik}^{(1,2)} \psi_{jk}^{(2,3)}  &\text{otherwise},
\end{cases}
\end{align*}
Furthermore, let $\sigma_{i,j,\la}\colon\eF_j\eE_i \id_\la \to \eE_i \eF_j \id_\la$ be given by 
\[\sigma_{i,j,\la} =(1_{\eE_i \eF_j}\otimes \epsilon)(1_{\eE_i}\otimes \psi_{ij} \otimes
1_{\eE_j})(\iota\otimes 1_{\eF_j\eE_i }).\] We also have the relation
\[\xi_{i,j,\la}^{-1}=
\begin{cases}
  \sigma_{i,j,\la} & i\neq j\\
  \begin{bmatrix}
    \sigma_{i,j,\la}\\
    \epsilon\\
y_i^{(1)}\epsilon\\
\vdots\\
(y_i^{(1)})^{\la^i-1}\epsilon
  \end{bmatrix}& i=j, \la^i\geq 0\\
  \begin{bmatrix}
    \sigma_{i,j,\la}&
    \iota &
\iota y_i^{(1)} &
\cdots &
\iota (y_i^{(1)})^{-\la^i-1}
  \end{bmatrix}& i=j, \la^i\leq 0\\
\end{cases}
\]
\end{definition}
As in \cite{KLIII}, we let $\dU$ denote the 2-category where every
Hom-category is replaced by its idempotent completion; we note that since
every object in $\tU$ has a finite-dimensional degree 0 part of its endomorphism algebra,
every Hom-category satisfies the Krull-Schmidt property.

This 2-category is a categorification of the universal enveloping
algebra is the sense that:
\begin{theorem}[\mbox{\cite[$\epsilon.8$]{WebCBerr}}]
  The graded  Grothendieck group of $\dU$ is isomorphic to $\dot{\bf U}^\Z_q$,
  Lusztig's integral modified quantum universal enveloping algebra.
\end{theorem}
The algebra $\dot{\bf U}_q$ can be thought of as $U_q(\fg)$ with additional idempotents $1_\la$
for integral weights $\la$
which satisfy the relations of projection to the $\la$-weight space.
The integral form $\dot{\bf U}^\Z_q$ is generated over $\Z[q,q^{-1}]$ by
$1_\la,E_i1_\la,F_i1_\la$ for all $\la$.
The map from the Grothendieck group sends the class of the 1-morphism $[\eE_i\colon \la\to
\la+\al_i]$ to $E_i1_\la$, and similarly for $\eF_i$.  
This theorem was first conjectured by Khovanov and
Lauda \cite{KLIII} and proven by them in the special case of
$\mathfrak{sl}_n$.  

\section{Quiver varieties}
\label{sec:quiver-varieties}

Recall that 
$\Gamma$ denotes the Dynkin graph of $\fg$.  
\begin{definition}
 For each
orientation $\Omega$ of $\Gamma$ (thought of as a subset of the edges
of the oriented double), a {\bf representation of
  $(\Gamma,\Omega)$ with shadows} is \begin{itemize}
\item a pair of finite dimensional $\C$-vector spaces $V=\oplus_{i\in \Gamma}V_i$ and $W=\oplus_{i\in \Gamma}W_i$,
  graded by the vertices of $\Gamma$, and
\item a map $x_e:V_{\omega(e)}\to V_{\al(e)}$ for each oriented edge
  (as usual, $\al$ and $\omega$ denote the head and tail of an
  oriented edge), and
\item a map $z:V\to W$ that preserves grading.
\end{itemize}
We let $\Bw$ and $\Bv$ denote $\Gamma$-tuples of integers.
\end{definition}  
For now, we fix an orientation $\Omega$, though we will sometimes wish
to consider the collection of all orientations.
With this choice, we
have the {\bf universal $(\Bw,\Bv)$-dimensional representation}
$$E_{\Bv,\Bw}=\bigoplus_{i\to  j}\Hom(\C^{v_i},\C^{v_j})
\oplus\bigoplus_i \Hom(\C^{v_i},\C^{w_i}).$$ 
In moduli terms, this is the moduli space of actions of the quiver (in
the sense above) on the vector spaces $\C^\Bv,\C^\Bw$, with their chosen bases considered as additional structure.

If we wish to consider the moduli space of representations where
$V$ has fixed graded dimension (rather than of actions on a fixed
vector space), we should quotient by the group of isomorphisms of
quiver representations: $G_\Bv=\prod_i\GL(\C^{v_i})$ acting by pre- and
post-composition.  The result is the {\bf moduli stack of $\Bv$-dimensional representations shadowed by $\C^\Bw$}, which we can
define as the stack quotient $$X^\Bw_\Bv=E_{\Bv,\Bw}/G_\Bv.$$ This is
not a scheme in the usual sense, but rather a smooth Artin stack.
Since we will only be interested in the constructible derived category of
sheaves on this stack, we do not need the full machinery of Artin
stacks, and could consider instead the equivariant
derived category of $E_{\Bv,\Bw}$ as in the book of Bernstein and Lunts \cite{BL} or as
described by the author and Williamson \cite{WWequ}.  We will always
consider this space as having the classical topology.

By convention, if $w_i=\al_i^\vee(\la)$ and $\xi=\la-\sum v_i\al_i$,
then $X^\la_\xi=X^\Bw_\Bv$ (if the difference is not in the positive
cone of the root lattice, then this is by definition empty), and $X^\la=\dot \sqcup_{\xi} X^\la_{\xi}$.

As mentioned in the introduction, our construction is inspired by the work of Li \cite{Li10} and
that of Zheng \cite{Zheng2008}.  Li defines a 2-category
built from
perverse sheaves on double framed quiver varieties. 
\begin{definition}
Li's 2-category is defined as follows:
\begin{itemize}
\item 0-morphisms are dimension vectors for the quiver $\Gamma$,
\item 1-morphisms between $\bd$ and $\bd'$ are objects of geometric
  origin in the
  localized derived category which Li denotes by
  $\mathscr{D}^-(E_{\Omega}(k^{\la},k^{\bd},k^{\bd'}))$, with product
  given by the convolution product of \cite[(16)]{Li10}.
\item 2-morphisms are morphisms in the category described above.
\end{itemize}
\end{definition}

For certain technical purposes, it is much more convenient for us to
use a different 2-category built using quantizations.   Let
\[\fM^\la_\xi=T^*E^\la_\xi/\!\!/_{\det}G_{\xi}=\xi^{-1}(0)^s/G_{\xi}\] be the Nakajima quiver
variety attached to $\la$ and $\xi$; this is a smooth, quasi-projective
variety which arises through geometric invariant theory as an open subset of the
cotangent bundle of $X^\la_\xi$.  See
\cite{Nak94,Nak98} for a more detailed discussion of the geometry of
these varieties.  
 
Any point in $T^*E^\la_\xi$ can be
thought of as a representation of the doubled quiver of $\Gamma$ (with
the framing maps also doubled).  The subset $\mu^{-1}(0)$ can be
thought of as parameterizing representations that descend to a certain
quotient of the doubled path algebra called the {\bf preprojective
  algebra}.  A particularly important result for us is a description
of the stable locus in terms of representation theory:
\begin{lemma}[\mbox{\cite[3.5]{Nak94}}]
  The subvariety $\mu^{-1}(0)^s$ is the subset whose associated
  preprojective representation has no non-trivial subrepresentation
  killed by all shadow maps.
\end{lemma}

Recall that a {\bf quantization} of the variety $\fM^\la_\xi$ as
defined in \cite{BK04a} or \cite[\S 3]{BLPWquant} is a sheaf $\A_\xi'$ of flat $\C[[h]]$-algebras with $\A_\xi'/h
\A_\xi'\cong \mathcal{O}_{\fM^\la_\xi}$ such the induced Poisson
structure on $\mathcal{O}_{\fM^\la_\xi}$ matches the standard
holomorphic symplectic structure on a quiver variety (induced from the
cotangent bundle $T^*E^\la_\xi$).  
For such a quantization, we let $\A_\xi=\A_\xi'[h^{-1}]$.


One method of constructing such quantizations is {\bf quantum Hamiltonian
reduction}. This operation was introduced in an algebraic context by
Crawley-Boevey, Etingof and Ginzburg \cite{CBEG} (though in many
contexts it appeared even earlier), and in the geometric form of
interest to us in \cite[2.8(i)]{KR07}.  Throughout we'll follow the
conventions of \cite{BLPWquant} for Hamiltonian reduction of
quantizations and refer the reader to constructions there (even those
which have appeared in older papers) in the interest of consistency.  We let $\mathcal{R}'$ be the
sheaf of microlocal differential operators on $T^* E^\la_\xi$, that
is, the Rees algebra for the usual order filtration of the sheaf 
$\D_{E^\la_\xi}$ of differential operators sheafified over $T^*E^\la_\xi$.  This algebra is
naturally a quantization of $T^*E^\la_\xi$ (see \cite[\S
4.1]{BLPWquant}).  We let $\mathcal{R}=\mathcal{R}'[h^{-1}]$; this is
a sheaf on $T^*E^\la_\xi$ such that taking pushforward under
$\pi\colon T^*E^\la_\xi$ we obtain $\pi_*\mathcal{R}\cong
\D_{E^\la_\xi}((h))$.  In particular, 
\[\Gamma(T^*E^\la_\xi;
\mathcal{R})\cong  \Gamma(E^\la_\xi;
\D_{E^\la_\xi})((h)).\]
Differentiating the action of $G$ on $E^\la_\xi$ induces a Lie algebra
map from $\fg$ to vector fields on $E^\la_\xi$, and thus a
non-commutative moment map $m\colon U(\fg)\to \mathcal{R}$.  Let
$\mathcal{R}_{\mathcal{S}}',\mathcal{R}_{\mathcal{S}}$ be the pullback
of these sheaves to the stable locus $\mathcal{S}$. As in \cite[\S 3.4]{BLPWquant}, we let \[\mathcal{E}=\mathcal{R} /\mathcal{R}
  m(\fg) \qquad \mathcal{E}_{\mathcal{S}}=\mathcal{R} _{\mathcal{S}}/\mathcal{R}_{\mathcal{S}} m(\fg) \] and consider the endomorphism
sheaf $\mathcal{E}nd_{\mathcal{R}
    _{\mathcal{S}}}(\mathcal{E}_{\mathcal{S}})$, which is naturally
  supported on $\mu^{-1}(0)^s$.  Let $p\colon \mu^{-1}(0)^s\to
  \fM^\la_\xi$ be the quotient map.
\begin{definition}
  We let $\A_\xi:=p_*\mathcal{E}nd_{\mathcal{R}
    _{\mathcal{S}}}(\mathcal{E}_{\mathcal{S}})$, the pushforward sheaf
  on $\fM^\la_\xi$.
\end{definition}
We actually have that $\A_\xi=\A_\xi'[h^{-1}]$ for a quantization 
$\A_\xi'$ obtained from $\mathcal{R}'$ by a similar reduction
procedure as in \cite[\S 3.4]{BLPWquant}.  The quantizations of
$\fM^\la_\xi$ can be classified by a cohomological invariant called
its {\bf period}.  This is a class in $hH^2(\fM^\la_\xi;\C)[[h]]$ and any
such class can be realized by a quantization since
$H^2(\fM^\la_\xi;\C)=H^{1,1}(\fM^\la_\xi;\C)$.  

On $X^\Bw_\Bv$, we have a tautological vector bundle $\mathscr{V}_i$ whose fiber over
a representation is $V_i$, the  part of that representation at node $i$; let $\mathscr{L}_i=\det(\mathscr{V}_i)$.
By \cite[6.4]{BLPWquant}, the period of $\A_\xi'$ is
\[\nicefrac{1}{2}\sum_{i\in \Gamma} \Big(w_i+\sum_{j\to i}v_j-\sum_{i\to j}
v_j\Big) c_1(\mathscr{L}_i)h.\]  Note that this period depends on the
choice of orientation of $\Gamma$, but its class modulo $hH^2(\fM^\la_\xi;\Z)$
does not.  Also, this is not always an integral class; this is a
generalization of the fact that differential operators, thought of as
a quantization of a cotangent bundle, do not always have integral
period (as \cite[3.10]{BLPWquant} shows).  If, as suggested in the
introduction, we think of the quiver variety as the cotangent bundle
of a hypothetical space $Y$, this would be the algebra of untwisted differential
operators on $Y$, and the quantization with period 0 would be the
differential operators in the square root of the canonical bundle of $Y$.

The quiver varieties carry a natural $\C^*$-action inherited from the
action on $T^*E_{\Bv,\Bw}$ scaling the cotangent fibers.  The sheaf of algebras
$\A_\xi$ carries a equivariant structure over $\C^*$ (see
\cite[2.3.3]{Losq}). We let $\A_\xi \mmod$
denote the category of $\C^*$-equivariant good modules over $\A_\xi$
(as defined in \cite[\S 4]{BLPWquant}).
 
The Hamiltonian reduction realization of $\A_\xi$ gives us a functor $\red:\D_{X^\la_\xi}\mmod\to
\A_\xi \mmod$ (called the ``Kirwan functor'' in \cite[\S
5.4]{BLPWquant}, where this functor is studied extensively) from
D-modules on $X^\la_\xi$ to $\A_\xi$-modules.
This functor proceeds by replacing a
D-module $\cM$ by its microlocalization $\bmu\cM:=\mathcal{R}\otimes_{\pi^{-1}\D_{X^\la_\xi}}\pi^{-1}\cM$, which is a sheaf on
$T^*X^\la_\xi\cong \mu^{-1}(0)/G$, and then restricting to $\fM^\la_\mu\subset
T^*X^\la_\xi$.  That is:
\begin{definition}
  The {\bf Kirwan functor} is the restriction $\displaystyle\red(\cM)=\cM|_{\fM^\la_\xi}$.
\end{definition}
While using stack language is elegant, 
one can also describe this in terms of the associated $G_\xi$-equivariant D-module $\cM'$ on
$E^\la_\xi$; this has microlocalization $\bmu\cM'$ supported on
$\mu^{-1}(0)$ by equivariance.  We restrict this to the stable locus, and take the invariant
pushforward  $\red(\cM)=p_*\mathcal{H}om_{\mathcal{R}_{\mathcal{S}}
}(\mathcal{E}_{\mathcal{S}},\bmu\cM'|_{\mathcal{S}})$, with its
natural $\A_\xi$-action.  Since on
$\mu^{-1}(0)^s$, the $G_\xi$ action is free, we always stay within the
world of varieties.  

Perhaps the most important property for us is that:
\begin{proposition}[\mbox{\cite[5.17]{BLPWquant}, \cite[1.1]{MNmorse}}]
  The functor $\red$ admits left and right adjoints
  \[\red_!\colon \A_\xi \mmod\to \D_{X^\la_\xi}\mmod\qquad \red_*\colon
  \A_\xi \mmod\to \D_{X^\la_\xi}\mmod,\] such that
  $\red\circ \red_!\cong \red\circ \red_*\cong\operatorname{id}$.
\end{proposition}

\begin{definition}
  We let $\EuScript{Q}^\la$ be the 2-category where
  \begin{itemize}
  \item 0-morphisms are dimension vectors for the quiver $\Gamma$,
  \item 1-morphisms between $\bd$ and $\bd'$ given by the
    bounded-below derived category of complexes of modules over
    $\A_\xi\boxtimes\A_{\xi'}^{op}$.  Composition of 1-morphisms
    $\cH_1\colon \fM^\la_{\xi_1}\to \fM^\la_{\xi_2}$ and $\cH_2\colon
    \fM^\la_{\xi_2}\to \fM^\la_{\xi_3}$ is given by convolution
    \begin{equation}\label{convolution}
      \cH_1\star\cH_2:=(p_{13})_*(p_{12}^*\cH_1\overset{L}\otimes_{\A_{\xi_2}}
      p_{23}^*\cH_2)[-\dim(\fM^\la_{\xi_1}\times \fM^\la_{\xi_3})].
    \end{equation}
  \item 2-morphisms are morphisms in the category described above; we
    consider this as a graded category with the homological grading.
  \end{itemize}
\end{definition}

This 2-category receives a natural 2-functor from the analytic version
of Li's 2-category; the (classical
topology) derived category of $X^\la_\xi$ has a functor to the derived
category of $\D_{X^\la_\xi}$-modules given by the Riemann-Hilbert
correspondence, and the functor $\red$ kills the necessary
subcategories to induce a 2-functor from the localization.  In order
to confirm that this is a 2-functor, we would have  check that we could also
define convolution as in Li's category \cite[\S 4.8]{Li10} (though,
his definition is ``dual'' to ours, since he uses the left, rather
than right adjoint of reduction).

\section{Hecke correspondences and categorical actions}
\label{sec:hecke-corr-categ}

We let $X^\la_{\xi;\nu}$ denote the moduli stack of short exact
sequences (``Hecke correspondences'') where the subobject belongs in
$X^\la_\xi$, the total object in $X^\la_{\xi-\nu}$ and the quotient in
$X^0_{-\nu}$. 

This moduli stack is naturally equipped with
projections
\begin{equation*}
  \begin{tikzpicture}[scale =2]
    \node(a) at (0,1) {$X^\la_{\xi;\nu}$};
    \node (c)at (0,0) {$X^\la_{\xi-\nu}$}; 
    \node (b) at (-1,0) {$X^\la_{\xi}$};    
    \node (d) at (1,0) {$X^0_{-\nu}$};
\draw[->,thick] (a) -- (b) node[midway,above left]{$p_1$};
\draw[->,thick] (a) -- (c) node[midway,right]{$p_2$};
\draw[->,thick] (a) -- (d) node[midway,above right]{$p_3$};
  \end{tikzpicture}
\end{equation*}
which we can think of more abstractly as taking the subobject, total
object and quotient, respectively.

If $\la-\xi=\sum v_i'\al_i$ and $\nu=\sum v_i''\al_i$, then we can
also think of this space more concretely.  We let 
\[E^\la_{\xi;\nu}\cong \bigoplus_{i\to
  j}\Hom(\C^{v_i'},\C^{v_j'})\oplus \Hom(\C^{v_i''},\C^{v_j'})\oplus
\Hom(\C^{v_i''},\C^{v_j''})\oplus\bigoplus_i \Hom(\C^{v_i},\C^{w_i})\]
and let $P_{\xi;\nu}$ be the parabolic in $G_{\Bv}$ which preserves
$\C^{v_i'}$ inside of $\C^{v_i}= \C^{v_i'}\oplus \C^{v_i''}$. 
We can alternatively define $X^\la_{\xi;\nu}:=E^\la_{\xi;\nu}/P_{\xi;\nu}$.

The projection maps are also easily understood from this perspective.
\begin{itemize}
\item The map $\pi_1$ is induced by the map
  $E^\la_{\xi;\nu} \to E^\la_{\xi}$ restricting each map to the
  subspace $\C^{v_i'}$ over each node.  
\item The map $\pi_2$ is induced by
  the inclusion $E^\la_{\xi;\nu}\hookrightarrow E^\la_{\xi-\nu}$
  induced by the isomorphism $\C^{v_i}= \C^{v_i'}\oplus \C^{v_i''}$.
\item The
  map $\pi_3$ is induced by the map $E^\la_{\xi;\nu} \to E^0_{-\nu}$
  which projects to $\Hom(\C^{v_i''},\C^{v_j''})$ for each
  arrow $i\to j$.
\end{itemize}
These maps are compatible with group homomorphisms
from $P_{\xi;\nu}$ to $G_{\Bv'},G_{\Bv},G_{\Bv''}$, and thus induce
maps of the appropriate quotient stacks.

For each $\la,\xi, i$, we
let 
\begin{equation}
 \tilde{\mathscr{F}}_i=\omega_{X^\la_{\xi-\al_i}}
\otimes_{\mathcal{O}_{X^\la_{\xi-\al_i}}}{(p_1\times
   p_2)}_*\mathcal{O}_{X^\la_{\xi;\al_i}}\qquad
    \mathscr{F}_i=\red(\tilde{\mathscr{F}}_i)\label{Fdef}
  \end{equation}
  \begin{equation}
\tilde  {\mathscr{E}}_i=\omega_{X^\la_{\xi}}\otimes_{\mathcal{O}_{X^\la_{\xi}}}{(p_2\times
      p_1)}_*\mathcal{O}_{X^\la_{\xi;\al_i}} \qquad 
  \mathscr{E}_i=\red(\tilde{ \mathscr{E}}_i)\label{Edef}
\end{equation}
Naturally,
    $\mathscr{F}_i$ is a module over
    $\A_\xi\boxtimes\A_{\xi-\al_i}^{op}$ and $\mathscr{E}_i$ is a
    module over $\A_{\xi-\al_i}\boxtimes\A_{\xi}^{op}$. That is, by
    definition, these are
    1-morphisms in $\EuScript{Q}^\la$ between the appropriate
    dimension vectors.  They are the images under the Riemann-Hilbert
    correspondence of the similarly named objects in Li's development
    of the theory.  We now proceed to our principal result:

\begin{theorem}\label{2-functor}
  We have a 2-functor of graded categories $\cG_\la\colon\tU\to \EuScript{Q}^\la$ sending
  $\eE_i\mapsto \mathscr{E}_i$ and  $\eF_i\mapsto \mathscr{F}_i$.
\end{theorem}
For now, we postpone the proof of this theorem to Section \ref{proof}, and instead discuss
its variations and  consequences in a bit more detail.

\subsection{The non-integral case}
\label{sec:non-integral-case}

The existence of the objects $ \mathscr{E}_i$ and $\mathscr{F}_i$ depends very strongly on the fact that
we use untwisted D-modules here.  Consider twists
\[\chi_1=\sum_{j\in \Gamma} a_ic_1(\mathscr{L}_j)\in
H^2(X^\la_\xi)\qquad \chi_2=\sum_{j\in \Gamma} b_ic_1(\mathscr{L}_j')\in
H^2(X^\la_{\xi-\al_i})\] on the respective varieties, where we use
$\mathscr{L}_j',\mathscr{V}_j'$ to denote the tautological bundles on
$X^\la_{\xi-\al_i}$.

\begin{proposition}\label{twist-exist}
  There exists a line bundle $\mathscr{L}$ on $X^\la_{\xi;\al_i}$ such
  that 
 \[\tilde{\mathscr{F}}_i^{\mathscr{L}}=\omega_{X^\la_{\xi-\al_i}}
\otimes_{\mathcal{O}_{X^\la_{\xi-\al_i}}}{(p_1\times
   p_2)}_*\mathscr{L}
\]
\[ \tilde  {\mathscr{E}}_i^{\mathscr{L}}=\omega_{X^\la_{\xi}}\otimes_{\mathcal{O}_{X^\la_{\xi}}}{(p_2\times
      p_1)}_*\mathscr{L}^{-1}\]
are bimodules over $\D_{X^\la_{\xi}}(\chi_1)$ and
$\D_{X^\la_{\xi-\al_i}}(\chi_2)$ if and only if $a_i,b_i,a_j-b_j\in
\Z$ for all $j\in \Gamma$. 
\end{proposition}
\begin{proof}
  First we note that $X^\la_{\xi},X^\la_{\xi-\al_i}$ and
  $X^\la_{\xi;\al_i}$ are each the quotient of an affine space by an
  affine algebraic group.  These groups are $G_\xi$ and
  $G_{\xi-\al_i}$ and a maximal parabolic in the latter, respectively.  Thus the Picard groups of these spaces are naturally
  identified with the character group of the group in question. In practice,
  this means that
  \begin{itemize}
  \item $\{c_1(\mathscr{L}_j)\}_{j\in \Gamma}$ is a basis of
    $H^2(X^\la_{\xi};\Z)$, 
\item $\{c_1(\mathscr{L}_j')\}_{j\in \Gamma}$ is
    a basis of $H^2(X^\la_{\xi-\al_i};\Z)$ and
\item
    $\{c_1(p_1^*\mathscr{L}_j)\}_{j\in \Gamma}\cup
    \{c_1(p_2^*\mathscr{L}_i')\}$ for $H^2(X^\la_{\xi;\al_i};\Z)$.
  \end{itemize}

In order to have the desired left and right twisted D-module structure, we must have that
  $c_1(\mathscr{L} )=p_1^*\chi_1-p_2^*\chi_2$; by the identification of the Picard group
  with homology, such an $\mathscr{L} $ exists if and only if $p_1^*\chi_1-p_2^*\chi_2\in
  H^2(X^\la_{\xi;\al_i};\Z)$.

  For $j\neq i$, we have that
  $p_1^*c_1(\mathscr{L}_j)=p_2^*c_1(\mathscr{L}_j')$, but for $i$,
  these are independent classes.
  Thus, we have that \[p_1^*\chi_1-p_2^*\chi_2=a_i
  p_1^*c_1(\mathscr{L}_i)-b_ip_2^*c_1(\mathscr{L}_i')+\sum_{j\neq i}
  (a_j-b_j)p_1^*c_1(\mathscr{L}_j),\] which is integral if and only if
  $a_i,b_i,a_j-b_j\in \Z$.
\end{proof}

Thus more generally, using Proposition \ref{twist-exist}, we can define
such an action where we choose any quantization corresponding to
differential operators in a line bundle on each $X^\la_\xi$, not just
the particular one we have fixed. 
If we instead choose a not necessarily integral twist $\chi=\sum
a_ic_1(\mathscr{L}_i)$, we only know at the moment how to construct a
categorical action of the smaller Lie algebra generated by the simple
root spaces where $a_i$ is integral.

This observation is particularly interesting in the case where
$\mathfrak{g}$ is affine, $\la$ is the basic fundamental weight and
$\xi=n\delta$.  In this case, the $\C^*$-invariant section algebra
$\Gamma(\fM^\la_\xi;\A_\xi)^{\C^*}$ is a spherical symplectic
reflection algebra for $S_n\wr \gamma$, where $\gamma$ matches $\fg$
under the Mackay correspondence by \cite{EGGO,Gorrem,Losq}. This
phenomenon of functors associated to roots appearing when particular
functions on the parameter space are integral is quite suggestive in
connection with Etingof's conjecture relating finite dimensional
modules for these symplectic reflection algebras to affine Lie
algebras \cite{EtiSRA}.  In fact, since the first version of this
paper appeared as a preprint, Bezrukavnikov and Losev \cite{BLet} have
proven this theorem using the functors that arise this way.

\subsection{Harish-Chandra and core modules}
\label{sec:harish-chandra-core}

The 2-functor $\cG_\la$ actually lands in a much smaller subcategory of
$\EuScript{Q}^\la$. In the product $\fM^\la_\xi\times
\fM^\la_{\xi'}$ we still have a notion
of ``diagonal.'' By \cite[3.27]{Nak98}, the affinization of a quiver
variety $\fN^\la_\xi$ lies in the moduli space of
semi-simple representations of the pre-projective algebra of a given
dimension.  We say a pair of such representations lies in the {\bf stable diagonal}
if they become isomorphic after the addition of trivial
representations (that is, they are isomorphic up to stabilization). 
In more concrete terms, the global functions on $\fN^\la_\xi$ are
generated by trace of the composition of the maps along a path in the
doubled Crawley-Boevey quiver\footnote{The Crawley-Boevey quiver
is $\Gamma$ with an additional vertex $\infty$ and $w_i$ new edges
attaching $i$ to $\infty$. We can think of an element of $E^\la_\xi$
as a representation of this quiver with $\C$ placed on $\infty$, and
thinking of each row in the matrix of the map $\C^{v_i}\to \C^{w_i}$
as the map along a different edge.} of $\Gamma$.  We can define the stable
diagonal to be the pairs in $\fN^\la_\xi\times
\fN^\la_{\xi'}$ where these traces agree for any path.

Following Nakajima, we let $Z$ denote the preimage of the stable
diagonal in $\fM^\la\times \fM^\la$; this can also be thought of as
the points where all the traces of loops coincide.  In \cite[\S 6.1]{BLPWquant}, Braden, Proudfoot and the
author define a 2-subcategory $\mathbf{HC}^g(\la)$ of good sheaves of
$\A_\xi\boxtimes \A_{\xi'}$-modules called {\bf Harish-Chandra
  bimodules}.  This is the category of modules $\cM$ such that:
\begin{itemize}
\item the support of $\cM$ is contained in $Z$.
\item there is a $\A'_\xi\boxtimes \A'_{\xi'}$-lattice $\cM'\subset
  \cM$ such that any global function vanishing on the stable diagonal
  kills the coherent sheaf $\cM'/h\cM'$.  This is a condition which
  should be thought of as an analogue of regularity of D-modules.
\end{itemize}

\begin{proposition}
 The image of  $\cG_\la$ lies in the 2-category $\mathbf{HC}^g(\la)$. 
\end{proposition}
\begin{proof}
Since $\mathbf{HC}^g(\la)$ is closed under convolution, we need only
check these conditions for $\mathscr{E}_i$ and $\mathscr{F}_i$.  We
have already checked that the supports of these modules are Hecke
correspondences, and thus lie in $Z$.  

  Furthermore, the D-modules $\tilde{\mathscr{E}}_i$ and
  $\tilde{\mathscr{F}}_i$ are the pushforwards of regular D-modules,
  and thus themselves regular. The corresponding very good filtrations
  on these D-modules have associated graded killed by any global
  function which vanishes on their support.  Note that global
  functions on $\fM^\la_\xi \times \fM^\la_{\xi\pm \al_i}$ are the same as
  invariant functions on $E^\la_\xi\times E^\la_{\xi\pm \al_i}$. Since any invariant
  function whose reduction vanishes on the stable diagonal must vanish
  on 
  the support of $\tilde{\mathscr{E}}_i$ and
  $\tilde{\mathscr{F}}_i$, it acts trivially on their associated
  graded.  Thus, it also acts trivially on the induced lattice on
  $\mathscr{E}_i$ or $\mathscr{F}_i$, and we are done.
\end{proof}

This draws an analogy between the categorical action $\cG_\la$ and the action
of the monoidal category of Harish-Chandra bimodules (in the classical
sense) on various categories of representations of $\fg$.  The latter
is a categorification of the Hecke algebra, which has
\begin{itemize}
\item its original representation-theoretic description, 
\item a geometric
  one via the localization theorem of Beilinson and Bernstein
  \cite{BB}, and 
\item a diagrammatic description in the guise of Soergel bimodules
  given by the work of Elias and Khovanov in type A \cite{ElKh} and
  work of Elias and Williamson in general \cite{ElW}.
\end{itemize}
The 2-category $\tU$ was first
defined in a purely diagrammatic manner, so it is striking evidence of
its naturality (at least to the author) to see it arise in a geometric
context as well.

This observation also has applications in practice.  Consider a system of subvarieties $J_\xi\subset
\fM^\la_\xi$
which is closed under convolution with $Z$.  Since the support of the
convolution of two modules is contained in the convolution of their supports, the
category of modules supported on $J_\xi$ is closed under this
categorical action.  Examples include:
\begin{itemize}
\item the {\bf cores} $L^\la_\xi$ of the varieties $\fM^\la_\xi$; that is, the
  subvariety of representations which are nilpotent as representations
  of the preprojective algebra.  Alternatively, the core $L^\la_\xi $
  is the preimage of the unique fixed point of the conic $\C^*$-action
  on $\fN^\la_\xi$.  

A sheaf of $\A_\xi$-modules is
  called a {\bf core module} if it is 
  supported on the core.  Let $\mathcal{C}^\la_\xi$ be the category of
  core modules on $\fM^\la_\xi$ for our fixed quantization $\A_\xi$,
  and $\mathcal{C}^\la:=\oplus_\xi
  \mathcal{C}^\la_\xi$.
\item the points attracted to the core under a $\C^*$-action for which
  the symplectic form has positive weight.
  The modules supported on these subvarieties (subject to a regularity
  condition like $\mathbf{HC}^g$) are an analogue of category $\cO$
  and are studied in much greater detail by Braden, Licata,
  Proudfoot and the author in \cite{BLPWgco}.
\end{itemize}

Thus, we have that:
\begin{corollary}
  The sum $\mathcal{C}^\la$ carries a categorical $\fg$-action.\qed
\end{corollary}
In fact, we can prove something stronger here.  Recent work of
Baranovsky and Ginzburg \cite{BaGi} shows that number of simples in $
\mathcal{C}^\la_\xi$ is less than or equal to the dimension of
cohomology group 
$H^{mid}(\fM^\la_\xi;\C)$. By work of Nakajima \cite{Nak98}, $\dim
H^{mid}(\fM^\la_\xi;\C)=\dim (V_\la)_\xi$, the weight multiplicity of $\xi$ in the simple
$\fg$-representation $V_\la$ with highest weight $\la$.  Let
$K(\mathcal{C}^\la)$ be the Grothendieck group of the abelian category
of $\C^*$-equivariant good core modules.
\begin{theorem}
  There is an isomorphism of $\fg$-representations $K(\mathcal{C}^\la)\cong V_\la$.
\end{theorem}
We should emphasize that this is only true in the case where the
quantization is integral (it corresponds to D-modules on an honest
line bundle).  There is always an injective map $K(\mathcal{C}^\la)\to
V_\la$ given by characteristic cycles (see \cite[6.2]{BLPWquant} or
\cite{KSdq}), but outside of the integral case, it seems to never to
be surjective.
Recent work of
Bezrukavnikov and Losev \cite{BLet} has calculated the structure of
this Grothendieck group in non-integral cases for finite type, and
certain especially important affine cases.  

\begin{proof}
  The space $K(
  \mathcal{C}^\la)$ is an $\fg$-representation, which has a weight
  decomposition by the definition of a categorical
  $\fg$-action.  Furthermore, by the result from \cite{BaGi}
  referenced above, the weight multiplicities of this representation
  are no more than those of the simple $V_\la$. 

  On the other hand $\A_\la\cong\C$, thought of as a sheaf on a
  point.  Thus, $ \mathcal{C}^\la_\la$ is equivalent to the category
  of $\C$-vector spaces, so $K(
  \mathcal{C}^\la)$ has weight multiplicity $1$ for $\la$.  By our
  bound by $\dim (V_\la)_\xi$, the weight $\la$ is maximal among weights with non-zero
  multiplicity, since any higher weight corresponds to an empty quiver
  variety. Thus all vectors of weight $\la$ in $K(
  \mathcal{C}^\la)$ are highest weight
  vectors.  

  This shows that $V_\la$ occurs as a composition factor with
  multiplicity 1, and by the
  bound on weight multiplicity, there can be no others.
\end{proof}
One of the powerful aspects of categorical actions is that they
constrain the structure of a category.  In particular,
categorifications of simple representations are essentially unique, by
work of Rouquier \cite{Rou2KM}.
Attached to the weights $\la$ and $\xi$, there is an algebra
$R^\la_\xi$, the {\bf cyclotomic KLR algebra}, such that the
projective modules over this algebra categorify the simple
representation $V_\la$.  By this uniqueness result, we have the following:
\begin{theorem}\label{core-cyclotomic}
There is a semi-simple core module $C_\xi$ for each $\xi$ such that
$\Ext^\bullet(C_\xi,C_\xi)\cong R^\la_\xi$. 
\end{theorem}
In fact in \cite{Webqui}, we will show that the
cohomology of $\Ext^\bullet(C_\xi,C_\xi)$ is formal, so Morita
theory for dg-categories will imply that
$\Ext^\bullet(C_\xi,-)$ induces an equivalence of dg-categories
$R^\la_\xi\operatorname{-dgmod}\cong D^b(
  \mathcal{C}^\la_\xi)$.  
\begin{proof}
  The desired module is the sum \[C_\xi=\bigoplus_{\Bi} \eF_{i_1}\cdots \eF_{i_n}\A^\la_\la\] for
  $\Bi=(i_1,\dots,i_n)$ the set of all sequences such that
  $\xi+\al_{i_1}+\cdots+\al_{i_n}=\la$.  This is
  the Hamiltonian reduction of the D-module $\oplus_{\Bi}
  \tilde{\mathscr{F}}_{i_1}\star\cdots\star
  \tilde{\mathscr{F}}_{i_n}\star \cO^\la_\la$ (where $\cO^\la_\la$ is
  the structure sheaf of the point $E^\la_\la$).  This D-module is
  semi-simple by the Decomposition Theorem, so the same is true of its
  reduction.  

By \cite[2.38]{Webmerged}, we have an isomorphism of the deformed
cyclotomic quotient $\check{R}^\la_\xi\otimes_{\check{R}^\la_\la}
\Ext^\bullet (C_\la,C_\la)\to \Ext^\bullet(C_\xi,C_\xi)$.  Since
$C_\la=\A^\la_\la$, this tensor product is just $R^\la_\xi$, and we
have the desired isomorphism $\Ext^\bullet(C_\xi)\cong R^\la_\xi$.

Thus the non-isomorphic simple summands of $C_\xi$ are in bijection with
indecomposable projectives $R^\la_\xi$.
We know from \cite[2.29]{Webmerged} that the K-group of $R^\la_\xi$
has dimension given by the weight multiplicity of $\xi$ in $V_\la$, so
this is the number of non-isomorphic simple summands of $C_\xi$.  By the upper bound of
  Baranovsky and Ginzburg, every core simple must be a summand of this module.
\end{proof}
Note that this equivalence matches the indecomposable projective
modules over $R^\la_\xi$ to the simple modules in
$\mathcal{C}^\la_\xi$. Since the main result of \cite{VV} matches
these indecomposables with Lusztig's canonical basis, we have that:
\begin{corollary}
  The isomorphism $K^0(\mathcal{C}^\la_\xi)\cong V_\la$ matches the
  classes of simples and Lusztig's canonical basis.  \qed
\end{corollary}
In particular, this shows that cyclotomic KLR algebras and Lusztig's
canonical basis for a simple have a natural
geometric origin based on quiver varieties.
Core modules may not seem like a familiar object to most readers, but
they are closely linked to finite dimensional modules over the section
algebra $A_\xi=\Gamma(\fM^\la_\xi;\A_\xi)^{\C^*}$.  It follows from
\cite[Th. B.1]{BLPWquant} that:
\begin{theorem}
  Every core module $\cM$ has a finite dimensional space of $\C^*$-invariant
  sections $\Gamma(\fM^\la_\xi;\cM)^{\C^*}$.  Furthermore, there
  exist choices of integral period such that
  $\Gamma(\fM^\la_\xi;-)^{\C^*}$ is an equivalence of categories
  between core modules and finite dimensional $A_\xi$-modules. \qed
\end{theorem}

Even when this sections functor fails to be an equivalence, it is
often a {\it derived} equivalence; the set of such periods actually
contains a Zariski open set.  The paper \cite{BLPWquant} contains a
much more detailed discussion of when localization and derived
localization hold.

While it does not follow from such a simple uniqueness argument, one
can generalize Theorem \ref{core-cyclotomic} to one connecting
category $\cO$'s to the {\bf weighted KLR
  algebras} introduced in \cite{WebwKLR}. This is proven in
in \cite[Th. A]{Webqui}.

\subsection{Canonical bases}
\label{sec:canonical-bases}

In this subsection, we assume that $\Gamma$ is an ADE Dynkin diagram.  

Another canonical basis worth considering is that for the modified
quantum universal enveloping algebra $\dot{\bf{U}}$.  By
\cite[Th. A]{WebCB}, this canonical basis coincides with the classes
of the indecomposable 1-morphisms in $\tU$.  

\begin{lemma}
  The 2-functor $\mathcal{G}_\la$ is full on 2-morphisms, that is for
  any 1-morphisms  $u$ and $v$, the map $\Hom_{\tU}(u,v)\twoheadrightarrow
  \Hom_{\EuScript{Q}_\la}(
 \mathcal{G}_\la( u), \mathcal{G}_\la (v))$.  
\end{lemma}
\begin{proof}
We induct downward on the usual order on
the weight lattice generated by $\mu-\al_i<\mu$.
  
We
    have that $\fM^\la_\la$ is a
    point, so the only non-trivial 1-morphism is the identity, and its
    endomorphisms are just the scalars.  In this case, fullness is
    clear.  This establishes the base case.
 
    Assume that we know the theorem for 1-morphisms $\mu'\to \nu'$
    where either $\mu'>\mu$ or $\nu'>\nu$.  Assume that $u$ and $v$
    are indecomposable. 
Recall that $\tU$ has a ``triangular decomposition'' into two
subcategories $\tU^+$ and $\tU^-$ generated by the $\eE_i$'s and
$\eF_i$'s respectively.
 We now prove two smaller claims:
    \begin{enumerate}
    \item if $v$ is not in the image of $\tU^-$, then
      $\Hom_{\tU}(u,v)\twoheadrightarrow \Hom_{\EuScript{Q}_\la}(\mathcal{G}_\la
      u,\mathcal{G}_\la v)$.
    \item if $u$ is not in the image of $\tU^+$, then
      $\Hom_{\tU}(u,v)\twoheadrightarrow \Hom_{\EuScript{Q}_\la}(\mathcal{G}_\la
      u,\mathcal{G}_\la v)$.
    \end{enumerate}
    Let us first consider (1).  If $v$ is not in the image of $\tU^-$
    then by \cite[5.12]{WebCB}, we have that $v$ is a summand of
    $\eE_iv'$ for some 1-morphism $\mu+\al_i\to \nu$; let
    $e\colon \eE_iv'\to \eE_iv'$ by an idempotent whose image is $v$,
    and $v''$ be the image of $1-e$, that is the complementary
    summand.  By assumption, we have a surjection
    \[\Hom_{\tU}(u,\eE_iv')\cong
    \Hom_{\tU}(\eF_iu,v')\twoheadrightarrow
    \Hom_{\EuScript{Q}_\la}(\mathcal{G}_\la (\eF_i u),\mathcal{G}_\la(v'))\cong
    \Hom_{\EuScript{Q}_\la}(\mathcal{G}_\la (u),\mathcal{G}_\la( \eE_i v')).\]
    With we compose this map with the idempotent $\mathcal{G}_\la e$, then we
    obtain a surjection
    $\Hom_{\tU}(u,\eE_iv')\twoheadrightarrow
    \Hom_{\EuScript{Q}_\la}(\mathcal{G}_\la (u),\mathcal{G}_\la (v))$,
    which kills $\Hom_{\tU}(u,v'')$; thus, the induced map
    $\Hom_{\tU}(u,v )\to 
    \Hom_{\EuScript{Q}_\la}(\mathcal{G}_\la (u),\mathcal{G}_\la (v))$
    is surjective as desired.  Claim (2) follows by a symmetric
    argument.

    Thus, it remains to establish
    $\Hom_{\tU}(u,v)\twoheadrightarrow \Hom_{\EuScript{Q}_\la}(\mathcal{G}_\la 
    u,\mathcal{G}_\la v)$
    for $u$ in the image of $\tU^-$ and $v$ in the image of $\tU^+$.
    For reasons of weight, the target can only be non-zero if $\mu=\nu$ and
    $u=v=\id_\mu$.  Thus, we must prove that
    \begin{equation}
    \Hom_{\tU}(\id_\mu,\id_\mu)\twoheadrightarrow
    \Hom_{\EuScript{Q}_\la}(\id_\mu,\id_\mu)\cong
    H^*(\fM^\la_\mu).\label{eq:Kirwan}
  \end{equation}
This surjectivity is a consequence of the ``algebraic Kirwan
surjectivity'' discussed in \cite{Webcenter}.  Combining the
surjectivity \cite[\ref{center-prop:Kirwan}]{Webcenter} to the center
of the cyclotomic quotient and the
isomorphism \cite[\ref{center-thm:final}]{Webcenter} of said center to
the cohomology of the quiver variety, the map of \eqref{eq:Kirwan}
must be surjective.
  \end{proof}
  By a standard argument (see, for example,
  \cite[Lemma \ref{compare-lem:full-indec}]{Webcomparison}), this shows that the
  functor $\mathcal{G}_\la$ sends each indecomposable 1-morphism to an
  indecomposable bimodule.

In fact, we can strengthen this statement:

\begin{lemma}\label{simples}
  For every simple $\D_{X^\la_\xi}^{\mbox{}}\boxtimes
\D_{X^\la_{\xi'}}^{op}$-module $L$, the reduction $\red(L)$ is simple,
and every simple object in  the heart of $\EuScript{Q}$  is a
reduction of such a simple module.
\end{lemma}
\begin{proof}
The functor $\red$ is exact; thus for any simple $K$ in $\EuScript{Q}$,
all but one composition factor of $\red_!(K)$ must be killed by
$\red$.  Thus, $K$ is the reduction of that composition factor $L$.  On the other
hand, if $L$ is simple and $\red(L)\neq 0$, then we have a non-zero map
$K\to \red(L)$ for some $K$, and thus a map $\red_!(K)\to L$.  Thus,
$L$ must be the unique composition factor of $\red_!(K)$ not killed by
$\red$, and so $\red(L)=K$ and is thus simple.
\end{proof}

Fix a vertex $i$ (which we assume to be a source).  It will frequently
be useful be useful to consider a variety intermediate to imposing all
stability conditions and none of them:
\begin{definition}\label{def:hat}
  We let $\hat{X}^\la_\xi$ be the open locus in $X^\la_\xi$ where the
  sum \[x_{out}\colon V_i\to W_i\oplus \bigoplus_{\alpha(e)= i}V_{\omega(e)}\]  of the maps along edges pointing out from $i$ is
  injective.  We let $\hat{X}^\la_{\xi\pm\al_i;\al_i}$ be the
  restriction of the correspondence ${X}^\la_{\xi\pm\al_i;\al_i}$ to
  the same locus.
\end{definition}

\begin{lemma}\label{semisimple}
  The object $\cG_\la(P)$ for $P$ a 1-morphism in $\tU$ is isomorphic
  to $\red_!(M)$ for $M$ a sum of shifts of simple regular holonomic
  D-modules and is thus a sum of shifts of simple $\A$-modules. 
 If $\tilde{\psi}(P)=P$, then the D-module $M$ can be taken to be self-dual.
\end{lemma}
We should note that this is an analogue of Conjecture 4.13 in
\cite{Li10} in our situation.

\begin{proof}
  In order to show both of these statements, we need only show that
  for any sequence $(\Bi)$ (including both positive and negative
  simple roots), the complex $\cG_\la(\Bi)$ is the reduction of a self-dual sum of
  shifts of regular holonomic D-modules. We induct on the length of
  $\Bi$.

First, we note that the identity 1-morphism, the sheaf $\A_\xi$ with the
diagonal bimodule structure, is simple since it has irreducible
support and the diagonal bimodule over the Weyl algebra is simple.
It is the analogue of this point which is actually quite difficult
in Li's category, and thus is an obstruction to using the techniques
described here in that situation.

 We have the
   diagram of maps 
   \begin{equation}
\tikz[very thick,baseline]{\node (a) at (0,-2) {$X^\la_\xi$}; 
   \node (b) at (0,0) {$\hat{X}^\la_\xi$};   
   \node (la) at (-5,-2) {$X^\la_{\xi-\al_i}$};
   \node (lb) at (-5,0) {$\hat{X}^\la_{\xi-\al_i}$};
   \node (mlb) at (-2.5,1.5) {$\hat{X}^\la_{\xi-\al_i;\al_i}$};
  \node (ra) at (5,-2) {$X^\la_{\xi+\al_i}$};
  \node (rb) at (5,0) {$\hat{X}^\la_{\xi+\al_i}$};
   \node (mrb) at (2.5,1.5) {$\hat{X}^\la_{\xi;\al_i}$};
\draw [->] (b) -- (a) node[midway,right]{$\iota$};
\draw [->] (mlb) --(b) node[above, midway]{$f_2$};
\draw [->] (mlb) --(lb) node[above, midway]{$f_1$};
\draw [->] (mrb) --(b) node[above, midway]{$e_1$};
\draw [->] (mrb) --(rb) node[above, midway]{$e_2$};
\draw [->] (lb) -- (la) node[midway,right]{$\iota_-$};
\draw [->] (rb) -- (ra) node[midway,right]{$\iota_+$};
   }\label{diagram}
 \end{equation}
We note that each one of these maps is smooth or an open inclusion.
We let \[f_1^{!*}:=f_1^*[\dim \hat{X}^\la_{\xi-\al_i;\al_i}- \dim
X^\la_{\xi-\al_i}]=f^!_1 [\dim X^\la_{\xi-\al_i}-\dim
\hat{X}^\la_{\xi-\al_i;\al_i} ]\] denote the unique shift of the
pullback functor with commutes with Verdier duality, and similarly for
$f_2^{!*},e_1^{!*},e_2^{!*}$.
We can
  consider a simple D-module $L$
  on $X^\la_\xi$.  Rewriting convolution with  $\hat{\mathscr{F}_{i}},
    \hat{\mathscr{E}_{i}}$ in terms of the diagram \eqref{diagram}, we
    have that \[\red(L\star
    \hat{\mathscr{E}_{i}})=\red((\iota_+)_*(e_2)_*e_1^{!*}\iota^*L)\qquad \red(L\star
    \hat{\mathscr{F}_{i}})=\red((\iota_-)_*(f_1)_*f_2^{!*}\iota^*L).\]
Furthermore, applying 
Lemma \ref{lem:hat-convolve}, 
  we have that
  \begin{equation}
\red(L)\star \mathscr{E}_i=\red((\iota_+)_*(e_2)_*e_1^{!*}\iota^*L)\qquad
\red(L)\star
\mathscr{F}_i=\red((\iota_-)_*(f_1)_*f_2^{!*}\iota^*L).\label{red-conv}
\end{equation}

By the Decomposition Theorem, the D-modules $(e_2)_*e_1^{!*}\iota^*L$ and
$(f_1)_*f_2^{!*}\iota^*L$ are a sum of shifts of simple D-modules on
  $\hat{X}^\la_{\xi\pm \al_i}$; furthermore, since all D-modules
  supported on $X^\la_{\xi\pm \al_i}\setminus \hat{X}^\la_{\xi\pm
    \al_i}$ are killed by $\red$, we can replace $(\iota_{\pm})_*$
  with the intermediate extension $(\iota_{\pm})_{!*}$ in \eqref{red-conv}:
\[\red(L)\star \mathscr{E}_i=\red((\iota_+)_{!*}(e_2)_*e_1^{!*}\iota^*L)\qquad
\red(L)\star \mathscr{F}_i=\red((\iota_-)_{*!}(f_1)_*f_2^{!*}\iota^*L).\]
Since intermediate extension preserves simplicity, we see that
$\red(L)\star \mathscr{E}_i$ is a reduction of sum of shifts of simple $\A_{\xi\pm
    \al_i}$-modules by Lemma \ref{simples}.
By the inductive assumption, $\cG_\la(\Bi)=\red(M)$ for $M$ a self-dual sum of shifts of
regular holonomic D-modules.  Thus $\cG_\la(\Bi,\pm i)$ is also a
reduction of  a self-dual sum of shifts of
regular holonomic D-modules (either
$(\iota_-)_{*!}(f_1)_*f_2^{!*}\iota^*M$ or
$(\iota_+)_{!*}(e_2)_*e_1^{!*}\iota^*M$) since the operations
$(\iota_-)_{*!}, \iota^*, (f_1)_*,f_2^{!*},(e_2)_*$ and $e_1^{!*}$ all
commute with duality (since $f_i,e_i$ are proper and smooth).
\end{proof}
Combining Lemma \ref{semisimple} with the observation that
indecomposability is preserved under this map, we
see that:
\begin{corollary}
  For an indecomposable 1-morphism $P$ in $\tU$, the sheaf
  $\cG_\la(P)$ is simple.
\end{corollary}
Let $\mathcal{Q}_\la$ denote the image of $\cG_\la$ as a functor
between graded additive categories, where the grading on the former
arises from the homological grading; this is a full
2-subcategory of $\EuScript{Q}_\la$, which is closed under convolution
(but {\it not} under extensions).  
This is a mixed humorous category in the sense of \cite[1.11]{WebCB}
by applying \cite[1.20]{WebCB} with $\EuScript{J}$ given by  the dg-subcategory $\cQ_\la$
generates equipped with the usual $t$-structure.   
Thus, if we let the {\bf canonical basis} of
$K_q(\mathcal{Q}_\la)$ be the classes of the simple modules, then
\cite[1.15]{WebCB} implies that these are also canonical bases in the
algebraic sense of bar-invariant almost-orthogonal vectors.  

Convolution also endows the graded Grothendieck group $K_q(\mathcal{Q}_\la)$ with an algebra, with
an induced algebra map $K_q(\mathcal{G}_\la)\colon K_q(\tU)\to
K_q(\mathcal{Q}_\la)$.  Finally, \cite[1.17]{WebCB} shows that:
\begin{proposition}
  Each canonical basis vector in $K_q(\mathcal{Q}_\la)$ is the
  image of a unique canonical basis vector in $\dot{\bf U}_q\cong K_q(\tU)$,
  and any other canonical basis vector in $\dot{\bf U}_q\cong
  K_q(\tU)$ in killed by $K_q(\mathcal{G}_\la)$.  
\end{proposition}

\subsection{Decategorification}
\label{sec:decategorification}

Finally, we turn to understanding how this action decategorifies.  As defined in \cite[\S 6.2]{BLPWquant}, based on work of Kashiwara and
Schapira \cite{KSdq}, we have a map
$\operatorname{CC}$ from the $K$-group of sheaves supported on $Z$ to  $H^{B\! M}_{top}(Z)$ which
intertwines convolution of sheaves with convolution of Borel-Moore
classes.  Composing the map induced on Grothendieck groups defined by
$\cG_\la$ with $\operatorname{CC}$, we obtain a homomorphism
 $C:K(\tU)\to H^{B\! M}_{top}(Z)$.   
\begin{proposition}
We have a commutative diagram 
\begin{equation}\label{triangle}
\tikz[->,very thick]{
\node (a) at (-3,0) {$K(\tU)$};
\node (b) at (3,0) {$H^{B\! M}_{top}(Z)$};
\node (c) at (0,1.5) {$\dot{\bf U}(\fg)$};
\draw (a) -- (b) node[below,midway]{$C$};
\draw (a) -- (c) node[above,midway,sloped]{$\sim$};
\draw (c) -- (b) node[above right ,midway]{$N$};
}
\end{equation}

  where 
  $N\colon \dot{\bf U}(\fg)\to H^{B\! M}_{top}(Z)$ is the map defined by Nakajima in \cite{Nak98}.
\end{proposition}
\begin{proof}
We can fix a vertex $i$, and assume that we have chosen our
orientation so that $i$ is a source.  Let $\hat{X}^\la_\xi$ and
$\hat{X}^\la_{\xi\pm\al_i;\al_i}$ be as defined in \ref{def:hat}.  We can define
  a bimodule $\mathscr{\hat{E}}_i$ restricting $\mathscr{\tilde{E}}_i$ to the hatted
  varieties.   We still have a reduction functor $\hat{\red}$ on D-modules over
  $\hat{X}^\la_\xi$, since any point where $x_{out}$ is not injective
  is destabilized by a subrepresentation on $i$ given by its kernel,
  and $\hat{\red}(\mathscr{\hat{E}}_i)=\mathscr{E}_i$.    

The map $\operatorname{CC}$ sends $[\mathscr{E}_i]$ to the sum of the 
fundamental classes of the components of its support variety, weighted by the generic
dimension of the stalk of its classical limit at a generic point of
the component.  However, the map $\hat{X}^\la_{\xi;\al_i}\to
\hat{X}^\la_\xi\times \hat{X}^\la_{\xi-\al_i}$ is injective with  smooth image; it is
locally modeled on the map from the $(k,k+1)$-type partial flag
variety to the 2 Grassmannians (with the ambient space given by the
sum $W_i\oplus \bigoplus_{\alpha(e)= i}V_{\omega(e)}$).  Thus, its pushforward has irreducible
characteristic variety with multiplicity one.  

The intersection of this characteristic variety with the stable locus
is the support variety of $\mathscr{E}_i$, which we can thus identify
with the Hecke
correspondence denoted $\mathfrak{P}_i$ in \cite{Nak98}. By a
symmetric argument,
the support variety of $\mathscr{F}_i$ is the variety obtained from this one by reversing
factors. 
  Thus, we have that \[ [\eE_i]\mapsto [\mathscr{E}_i]\mapsto
  [\mathfrak{P}_i]\qquad [\eF_i]\mapsto [\mathscr{F}_i]\mapsto
  [\omega(\mathfrak{P}_i)].\]
By \cite[9.4]{Nak98}, the homomorphism  $N\colon\dot U\to
H^{B\! M}_{top}(Z)$ is the unique one with this property.
\end{proof}

Similarly, in \cite[6.5.4]{KSdq}, it's shown than the action of $K(\tU)$ on
$K(\mathcal{C}^\la)$ by convolution is intertwined by $\operatorname{CC}$ with the Nakajima's
action of $U(\fg)$ on top Borel-Moore homology of the core $L^\la:=\sqcup_\xi L^\la_\xi$.  That is:
\begin{corollary}
  The diagram \eqref{triangle} can be extended to a commutative
  diagram including the natural actions of $K(\tU)$ on
  $K(\mathcal{C}^\la)$ induced by $\mathcal{G}_\la$, Nakajima's action
  of $H^{BM}_{top}(Z)$ on $H^{BM}_{top}(L^\la)$, and the usual action
  of $U(\fg)$ on $V_\la$:
  \[\tikz{\matrix[row sep=11mm,column sep=17mm,ampersand
    replacement=\&]{
      \node (a) {$K(\tU)$}; \& \node (c) {$H^{BM}_{top}(Z)$}; \&  \node (e) {$\dot{U}(\fg)$}; \\
      \node (b) {$K(\cC^\la)$};\& \node (d) {$H^{BM}_{top}(L^\la)$}; \&  \node (f) {$V_\la$};\\
    }; \draw[very thick,->] (a)--node[above,midway]{$C$} (c);
    \draw[very thick,->] (b)--node[above,midway]{$\operatorname{CC}$}
    (d); \draw[very thick,->] (e)--node[above,midway]{$N$} (c);
    \draw[very thick,->] (f)--node[above,midway]{$\sim$} (d);
    \draw[very thick,->] (b.70) to[out=70,in=0] (a.south) to[out=180,
    in=110] (b.110); \draw[very thick,->] (d.70) to[out=70,in=0]
    (c.south) to[out=180, in=110] (d.110); \draw[very thick,->] (f.70)
    to[out=70,in=0] (e.south) to[out=180, in=110] (f.110);}
  \hfill \qed \]
\end{corollary}

\section{The proof of Theorem~\ref{2-functor}}
\label{proof}

Now we proceed to the proof of Theorem~\ref{2-functor} through a
series of lemmata.
Let $m^\la_\mu=\dim \fM^\la_\mu$

\begin{lemma}\label{lem:biadjoint}
  The left and right  adjoint of 
    $\mathscr{F}_{i}\star-$ are the convolution functors \[\mathscr{E}_{i}[\nicefrac{1}{2}(m^\la_\mu -m^\la_{\mu-\al_i} )
  ] \star-\qquad \text{ and }\qquad \mathscr{E}_{i}[\nicefrac{1}{2}(m^\la_{\mu-\al_i}  -m^\la_\mu )
  ] \star-.\]
\end{lemma}
\begin{proof}
  As in \cite[(2.3.18)]{KSdq}, we can take the dual
  \[\mathbb{D}\mathscr{E}_{i}\cong \mathbb{R}\mathcal{H}om_{\A_{\mu}\boxtimes\A_{\mu-\al_i}^{op} }(\mathscr{E}_{i},\A_{\mu}\boxtimes\A_{\mu-\al_i}^{op}).\]
  This a left $\A_{\mu-\al_i}$-module and right $\A_\mu$-module
  By \cite[2.3.15]{KSdq}, its is the shift of a
  $\A_{\mu-\al_i}\boxtimes \A_\mu$ module with Lagrangian support
  equal to $\omega(\mathfrak{P}_i)]$.  Thus, we must have
  $\mathbb{D}\mathscr{E}_{i}[-\nicefrac{1}{2}(m^\la_\mu+m^\la_{\mu-\al_i})]\cong \mathscr{F}_{i}$; this is also easily
  shown using local computations with D-modules.  By
  \cite[6.2.4]{KSdq}, this homological shift is just convolution with
  the square root of the dualizing sheaf for $\A_{\mu-\al_i}\boxtimes \A_{\mu}$.
  We denote the dualizing sheaf for $\A_\mu$ by $\omega_{\mu}$. Similarly, $\mathbb{D}\mathscr{F}_{i}[-\nicefrac{1}{2}(m^\la_\mu+m^\la_{\mu-\al_i})]\cong \mathscr{E}_{i}$.
For any $\A_\mu$-module $\cM$ and
\nc{\cN}{\mathcal{N}}
$\A_{\mu-\al_i}$-module $\cN$, we have that
\begin{align*}
\mathbb{R}\mathcal{H}om_{\A_\mu}(\cM,\mathscr{F}_{i}\star\cN)&\cong
\mathbb{D}\cM\star (\mathscr{F}_{i}\star\cN)\\
&\cong
(\mathbb{D}\cM\star \mathscr{F}_{i})\star\cN & \text{\cite[3.2.4]{KSdq}}\\
&\cong (\mathbb{D}\cM\star \omega_{\mu-\al_i}^{\nicefrac{1}{2}}\star
  \mathbb{D}\mathscr{E}_{i}\star \omega_{\mu}^{\nicefrac{1}{2}})\star\cN \\
 &\cong \mathbb{D}(\mathscr{E}_{i}\star
  \cM)\star \cN[\nicefrac{1}{2}(m^\la_{\mu-\al_i} -m^\la_\mu)
  ] & \text{\cite[3.3.6]{KSdq}}\\
&\cong
  \mathbb{R}\mathcal{H}om_{\A_\mu}(\mathscr{E}_{i}\star\cM[\nicefrac{1}{2}(m^\la_\mu -m^\la_{\mu-\al_i} )
  ] ,\cN).
\end{align*}
By
  symmetry, this gives the biadjunction. 
\end{proof}

\begin{lemma} \label{lem:conv-adj}
  We have a natural isomorphism $M \star \red N\cong \red(\red_*M\star N)$ for any complex of
  $\A_\xi\boxtimes \A_{\xi'}^{op}$-modules $M$ and any complex of 
  $\D_{X^\la_{\xi'}}\boxtimes \D_{X^\la_{\xi''}}^{op}$-modules $N$.
\end{lemma}
\begin{proof}
  We need only confirm this isomorphism when $M=M_1\boxtimes
 M_2$ and $N=\D_{X^\la_{\xi'}}\boxtimes \D_{X^\la_{\xi''}}$.
  In this case, \[M \star \red N\cong M_1\boxtimes
  \Gamma(\fM^\la_{\xi'};M_2)\boxtimes
  \A_{\xi''} \qquad \red(\red_*M\star N)\cong M_1\boxtimes
 \Gamma(X^\la_{\xi'};\red_* M_2)\boxtimes
  \A_{\xi''} \]
The result then follows from the isomorphism
\[\Gamma(\fM^\la_{\xi'};M_2)\cong \Hom_{\A_{\xi'}
}(\A_{\xi'},M_1)\cong
\Hom_{\D_{X^\la_{\xi'}}}(\D_{X^\la_{\xi'}}, \red_*M_2)\cong \Gamma(X^\la_{\xi'};\red_* M_2).\qedhere\]
\end{proof}
\begin{lemma}\label{unstable-convolve}
  If $\cM$ is a $\D_{X^\la_{\xi}}$-module whose microsupport $\msupp(\cM)$ is contained in the unstable
  locus, then $\msupp(\cM\star
        \tilde{\mathscr{E}_{i}})$ is also contained in the unstable
        locus.   That is, if $\red(\cM)=0$, then $\red(\cM\star
        \tilde{\mathscr{E}_{i}})=0$.  
\end{lemma}
\begin{proof}
  Since the map $p_1\times p_2$ is a composition of a closed
inclusion and a smooth map, we can apply the description of the effect
of these maps on singular supports given in \cite[9a \& b]{BernD}.
Thus, the microsupport $\msupp(\tilde{\mathscr{E}}_{i})$ lies in the
image in $T^*X^\la_{\xi}\times T^*X^\la_{\xi-\al_i}$ of 
\[\big\{ x\in X^\la_{\xi;\al_i}, \varphi\in T^*_{(p_1(x),p_2(x))}(X^\la_{\xi}\times
X^\la_{\xi-\al_i}) \big\vert (p_1\times p_2)^*\varphi=0
\big\}.\] We can think of $x$
as a representation of the oriented quiver with chosen subrepresentation, the
covector $\varphi$ as a choice of maps along the oppositely oriented
arrows that extend the total representation and the subrepresentation
to representations of the preprojective algebra; in this case, the vanishing condition
is simply that the inclusion is a map of preprojective
representations. That is, the microsupport of
$\tilde{\mathscr{E}}_{i}$ is composed of the pairs of representations
of the preprojective algebra such that the LHS is isomorphic to a
subrepresentation of the RHS.
In particular, 
if the
lefthand point $p_1(x)$ has a destabilizing subrepresentation, 
$p_2(x)$ does as well.
That is, the microsupport of $\tilde{\mathscr{E}}_{i}$ has the property that if
        $p_1(x)$ is unstable, then $p_2(x)$ is as well. 

Applying the same result from \cite[9a \& b]{BernD}, the microsupport $\msupp(\cM\star
        \tilde{\mathscr{E}_{i}})$ is contained in the set 
\[\{(x,\varphi)\in T^*X^\la_{\xi-\al_i}\vert (x',\varphi';x,\varphi)\in
\msupp(\tilde{\mathscr{E}}_{i}) \text{ for some } (x',\varphi') \in
\msupp(\cM).\}\]  This shows that all points in this microsupport must
be unstable.
\end{proof}

 \begin{lemma}
    $\displaystyle \red(\tilde{\mathscr{E}_{i_1}}\star \cdots \star
    \tilde{\mathscr{E}}_{i_n})\cong \mathscr{E}_{i_1}\star \cdots \star
    \mathscr{E}_{i_n}$
  \end{lemma}
  \begin{proof}  We induct on $n$; when $n=1$, this is true by
    definition.

By the inductive hypothesis,  $\displaystyle \red(\tilde{\mathscr{E}_{i_1}}\star \cdots \star
    \tilde{\mathscr{E}}_{i_{n-1}})\cong \mathscr{E}_{i_1}\star \cdots \star
    \mathscr{E}_{i_{n-1}}$. Thus, we have maps \[\red_!(\mathscr{E}_{i_1}\star \cdots \star
    \mathscr{E}_{i_{n-1}})\to \tilde{\mathscr{E}_{i_1}}\star \cdots \star
    \tilde{\mathscr{E}}_{i_{n-1}}\to \red_*(\mathscr{E}_{i_1}\star \cdots \star
    \mathscr{E}_{i_{n-1}})\] which induce isomorphisms after
    applying $\red$.
  
We have a map $a\colon \red_!(\mathscr{E}_{i_1}\star \cdots \star
    \mathscr{E}_{i_{n-1}})\to \tilde{\mathscr{E}_{i_1}}\star \cdots \star
    \tilde{\mathscr{E}}_{i_{n-1}}$
which induces an isomorphism after
    applying $\red$.   Thus $C(a)$, the cone of this morphism, has
    cohomology microsupported on the unstable locus.  By definition,
    we have an exact triangle
    \begin{equation}
      \label{eq:1}
      \red_!(\mathscr{E}_{i_1}\star \cdots \star
    \mathscr{E}_{i_{n-1}}) \to \tilde{\mathscr{E}_{i_1}}\star \cdots \star
    \tilde{\mathscr{E}}_{i_{n-1}}\to C (a)\overset{[1]}{\to}
    \end{equation}

Thus, applying the triangulated functor $-\star
        \tilde{\mathscr{E}_{i_n}}$ to the equation \eqref{eq:1}, we have an exact triangle
        \begin{equation*}
\red_!(\mathscr{E}_{i_1}\star \cdots \star
    \mathscr{E}_{i_{n-1}}) \star
        \tilde{\mathscr{E}_{i_n}}\to \tilde{\mathscr{E}_{i_1}}\star \cdots \star
    \tilde{\mathscr{E}}_{i_{n}}\to C (a) \star
        \tilde{\mathscr{E}_{i_n}}\overset{[1]}{\to}.
      \end{equation*}

By Lemma \ref{unstable-convolve}, we have
\begin{math}
\red(C (a) \star
        \tilde{\mathscr{E}_{i_n}})=0
      \end{math}, so applying $\red$ to this exact triangle shows that 
      \begin{equation}
       \label{eq:2}
        \red( \red_!(\mathscr{E}_{i_1}\star \cdots \star
    \mathscr{E}_{i_{n-1}}) \star \tilde{\mathscr{E}}_{i_n})\cong \red(\tilde{\mathscr{E}_{i_1}}\star \cdots \star
    \tilde{\mathscr{E}}_{i_{n}})
      \end{equation}

 If we apply Lemma \ref{lem:conv-adj} with
    $M= \mathscr{E}_{i_1}\star \cdots \star
    \mathscr{E}_{i_{n-1}}$ and $N=\tilde{\mathscr{E}}_{i_n}$, we
    arrive at
    \begin{equation}
\mathscr{E}_{i_1}\star \cdots \star
    \mathscr{E}_{i_n}\cong \red( \red_!(\mathscr{E}_{i_1}\star \cdots \star
    \mathscr{E}_{i_{n-1}}) \star \tilde{\mathscr{E}}_{i_n}).\label{eq:3}
  \end{equation}

Combining equations (\ref{eq:2}--\ref{eq:3}), we arrive at the desired
        isomorphism 
\[\mathscr{E}_{i_1}\star \cdots \star
    \mathscr{E}_{i_n}\cong \red( \red_!(\mathscr{E}_{i_1}\star \cdots \star
    \mathscr{E}_{i_{n-1}}) \star \tilde{\mathscr{E}}_{i_n})\cong \red(\tilde{\mathscr{E}_{i_1}}\star \cdots \star
    \tilde{\mathscr{E}}_{i_{n}}).\qedhere\]
\end{proof}

\begin{lemma} \label{lem:hat-convolve} For any $\D_{X^\la_\mu}$-module $\cM$, we have that:
  \begin{equation}\label{ii-red}
    \red\cM\star
    \mathscr{F}_{i}\cong \red(\cM\star
    \hat{\mathscr{F}_{i}})\qquad    \red\cM\star
    \mathscr{E}_{i}\cong \red(\cM\star
    \hat{\mathscr{E}_{i}})
  \end{equation}
\end{lemma}
\begin{proof}
  We let
  $\hat{O}^\la_{\xi}=T^*\hat{X}^\la_{\xi}\setminus \fM^\la_{\xi}$
  and
  \[\hat{I}\colon \fM^\la_{\xi}\times\fM ^\la_{\xi+\al_i} \to
  T^*X^\la_{\xi}\times \fM^\la_{\xi-\al_i}\qquad \hat{J}\colon
  O^\la_{\xi}\times\fM ^\la_{\xi+\al_i} \to T^*\hat{X}^\la_{\xi}\times
  \fM^\la_{\xi-\al_i}\]
  be the inclusion of the loci where the the first coordinate is
  (un)stable. By definition,
  \[\red\cM\star
  \mathscr{F}_{i}\cong (p_2)_*\hat{I}_*\hat{I}^*
  (p_{1}^*\bmu\cM\otimes \bmu\hat{\mathscr{F}}_{i})\qquad
  \red(\cM\star \hat{\mathscr{F}_{i}})\cong (p_2)_*
  (p_{1}^*\bmu\cM\otimes \bmu\hat{\mathscr{F}}_{i})\]
  Thus, by the usual recollement, these will be isomorphic via the
  natural map if and only if
  $(p_2)_*\hat{J}_!\hat{J}^! (p_{1}^*\bmu\cM\otimes
  \bmu\hat{\mathscr{F}}_{i})=0$.
  Thus, it suffices to show that
  $\hat{J}^! (p_{1}^*\bmu\cM\otimes \bmu\hat{\mathscr{F}}_{i})=0$.
  Any point which lies in its support must have the following
  properties: its two coordinates correspond to framed modules
  $S_1,S_2$ over the preprojective algebra with
  an inclusion $S_2\hookrightarrow S_1$, such that $S_2$ is stable,
  and $S_1$ has a destabilizing subrepresentation $Z_1\subset S_1$.
  Furthermore, this destabilizing subrepresentation cannot lie solely
  on the vertex $i$.  However, the cokernel $S_1/S_2$ is only
  supported on $i$, so $Z_1$ cannot inject into this quotient.  The
  intersection $Z_1\cap S_2$ must thus be non-trivial, providing a
  destabilizing subrepresentation of $S_2$.  Thus, we have arrived at
  a contradiction, and this sheaf must have empty support, and thus be
  0. The second equation follows by a similar argument
  $  T^*X^\la_{\xi}\times \fM^\la_{\xi+\al_i}$, or alternately, from
  Lemma \ref{unstable-convolve}.
\end{proof}

\begin{proof}[Proof of Theorem~\ref{2-functor}]
We wish to check the conditions of \cite[4.13]{RouQH}.  This is
  defined by a list of conditions, which we check in the same order.
  \begin{itemize}
  \item the functors $\mathscr{E}_{i}\star-$ and
    $\mathscr{F}_{i}\star-$ are biadjoint up shift.  This follows from
    Lemma \ref{lem:biadjoint}. 
\item the sheaves 
  $\oplus_{\Bi}\mathscr{E}_{i_1}\star \cdots \star
    \mathscr{E}_{i_n}$ carry an action of the KLR algebra for the
    polynomials $Q$ we have specified. The solution sheaf (i.e. the image under
  the Riemann-Hilbert correspondence) of $\tilde{\mathscr{E}_{i_1}}\star \cdots \star
    \tilde{\mathscr{E}}_{i_n}$ is precisely the perverse sheaf that
    Varagnolo and Vasserot denote by ${}^{\delta\!}\mathcal{L}_{\Bi}$ in
    \cite{VV}.

In our language, \cite[3.5]{VV} and \cite[5.7]{RouQH} (independently) show that the Ext algebra of
  solution sheaves of $\oplus_{\Bi} \tilde{\mathscr{E}_{i_1}}\star \cdots \star
    \tilde{\mathscr{E}}_{i_n}$ is given by the KLR
    algebra $R= \oplus R_\nu$; since the Riemann-Hilbert correspondence is an
    equivalence of categories, we arrive at an isomorphism
\[\Ext^\bullet\big(\oplus_{\Bi} \tilde{\mathscr{E}_{i_1}}\star \cdots \star
    \tilde{\mathscr{E}}_{i_n}\big) \cong  R.\]  It follows that the image of these sheaves under any functor, in
    particular $\red(\tilde{\mathscr{E}_{i_1}}\star \cdots \star
    \tilde{\mathscr{E}}_{i_n})\cong {\mathscr{E}_{i_1}}\star \cdots \star
    {\mathscr{E}}_{i_n}$, still
    carry this action.
\item The functors $\mathscr{E}_{i}\star-$ and $\mathscr{F}_{i}\star-$
  are locally nilpotent.  This follows from that fact that for fixed $\xi$,
    there are only finitely many integers such that
    $\fM^\la_{\xi+k\al_i}$ is non-empty.
  \end{itemize}
 
The final condition is that for each $i$, we have \begin{samepage}
  \begin{equation*}\subeqn\label{item:1}
  {\mathscr{F}_{i}}\star
  {\mathscr{E}_{i}}\cong {\mathscr{E}_{j}}\star {\mathscr{F}_{i}}
  \oplus (q^{\langle\al_i,\xi\rangle +d_i}+\cdots
  +q^{-\langle\al_i,\xi\rangle -d_i}) \cdot \A_\xi\quad \text{ if
  }\langle\al_i,\xi\rangle\leq 0 
\end{equation*}
\begin{equation*}\subeqn\label{item:2}
  {\mathscr{F}_{i}}\star
  {\mathscr{E}_{i}} \oplus (q^{-\langle\al_i,\xi\rangle +d_i}+\cdots
  +q^{\langle\al_i,\xi\rangle -d_i}) \cdot \A_\xi\cong
  {\mathscr{E}_{j}}\star {\mathscr{F}_{i}} \quad\text{ if
  }\langle\al_i,\xi\rangle\geq 0. 
\end{equation*}
\end{samepage}
Once this is proven, the result will follow.
This can be proven from calculations done
on the level of constructible sheaves carried through the
Riemann-Hilbert correspondence. 
That these isomorphisms exist is shown in the characteristic $p$ setting by
\cite[1.12-13]{Li10b} following similar proofs of Zheng \cite{Zheng2008}. Li's proofs use no special facts about
characteristic $p$ fields; in principle, we could simply cite his
work, but for the sake of completeness, we give arguments in the
deformation quantization setting for the same facts.  

 Applying Fourier transform as necessary, we can
  assume that $i$ is a source.  Proposition \ref{twist-exist} assures
  us that we can transition between the different quantizations that
  arise from different orientations; the categories of modules over
  the different quantizations that arise are given by tensor product
  with quantizations of line bundles to bimodules given in
  \cite[5.2]{BLPWquant}. It's easily seen that these intertwine the
  actions of $\tU$. 

We consider
  the diagram (\ref{diagram}) of maps again.
Recall that the
  maps $e_i$ and $f_i$ are both proper and smooth; their fibers are
  projective spaces.  We let $\hat{\mathscr{F}}_i$ denote the
  restriction of $\tilde{\mathscr{F}}_i$ to the injective locus, and
  similarly for $\hat{\mathscr{E}}_i$.

We have an isomorphism of $\hat{\mathscr{E}_{i}}\star
  \hat{\mathscr{F}_{i}}$ with the pushforward by
  $\mathbbm{f}_2:=f_2\times f_2$ of the structure sheaf
  of
  $\hat{X}^\la_{\xi;\al_i}\times_{\hat{X}^\la_{\xi-\al_i}}\hat {X}^\la_{\xi;\al_i}$
  tensored with the canonical sheaf of the right factor.  The sheaf $\hat{\mathscr{F}_{i}}\star
  \hat{\mathscr{E}_{i}}$ is derived in the same way from $\mathbbm{e}_1: =e_1\times e_1$.

Now, we turn to showing that  
\newseq
\begin{equation*}\subeqn
\hat{\mathscr{F}_{i}}\star
  \hat{\mathscr{E}_{i}}\cong \hat{\mathscr{E}_{j}}\star
  \hat{\mathscr{F}_{i}} \oplus (q^{\langle\al_i,\xi\rangle
    +1}+\cdots +q^{-\langle\al_i,\xi\rangle
    -1}) \cdot \mathcal{D}_{ \Delta}\quad \text{ if }\langle\al_i,\xi\rangle\leq 0 \label{eq:4}
\end{equation*}
\begin{equation*}\subeqn
\hat{\mathscr{F}_{i}}\star
  \hat{\mathscr{E}_{i}} \oplus (q^{-\langle\al_i,\xi\rangle
    +1}+\cdots +q^{\langle\al_i,\xi\rangle
    -1}) \cdot \mathcal{D}_{ \Delta}\cong \hat{\mathscr{E}_{j}}\star
  \hat{\mathscr{F}_{i}} \quad\text{ if }\langle\al_i,\xi\rangle\geq
  0\label{eq:5}
\end{equation*}
where $\Delta$ denotes the diagonal in $ \hat{X}^\la_{\xi}\times
\hat{X}^\la_{\xi}$.
The analogous calculation for $\ell$-adic sheaves is done by
Li in \cite[1.13]{Li10b}; specifically, his equations
\cite[(19-20)]{Li10b} compute the two sides of the proceeding
displayed equations and show that they agree.

Since his proof is not especially difficult, let us give an account for the reader.  The maps \[\mathbbm{f}_2\colon
\hat{X}^\la_{\xi;\al_i}\times_{\hat{X}^\la_{\xi-\al_i}}\hat
{X}^\la_{\xi;\al_i}\to \hat{X}^\la_{\xi}\times\hat
{X}^\la_{\xi}\qquad \mathbbm{e}_1\colon
\hat{X}^\la_{\xi+\al_i;\al_i}\times_{\hat{X}^\la_{\xi+\al_i}}\hat
{X}^\la_{\xi+\al_i;\al_i}\to \hat{X}^\la_{\xi}\times\hat
{X}^\la_{\xi}\] both have image given by the set $H^\la_\xi$ of representations
(injective at $i$) which are the same away from $i$, and where the
subspaces at $i$ have intersection of codimension 1 in both spaces.
Both $\mathbbm{f}_2$ and $\mathbbm{e}_1$ induce an isomorphism on the
locus in $H^\la_\xi$ where the representations differ,
and a projective space bundle at the points where they coincide.  

The
map $\mathbbm{e}_1$ has fiber over the diagonal given by
$\mathbb{P}^{v_i-1}$, since the fiber consists of all the ways of
choosing a hyperplane in $V_i$.
 The
map $\mathbbm{f}_2$ has fiber $\mathbb{P}^{\langle\al_i,\xi\rangle+v_i-1}$, since the
fiber consists of the lines in the cokernel of $x_{out}$.  Finally, the diagonal
has  codimension $\langle\al_i,\xi\rangle+2v_i-1$ inside $H^\la_\xi$.  Thus, the map
$\mathbbm{e}_1$ is small if $\langle\al_i,\xi\rangle\geq 0$, and the map $\mathbbm{f}_2$
is small if $\langle\al_i,\xi\rangle\leq 0$.  Let's reduce to the former case for
simplicity.

In this case, $\hat{\mathscr{F}_{i}}\star
  \hat{\mathscr{E}_{i}}$ is an 
  irreducible D-module $Q$, the unique one on $H^\la_\xi$ which extends the pullback of
  the canonical sheaf by the first projection, since it is the pushforward by a small
  resolution of singularities.  On the other hand, $\hat{\mathscr{E}_{i}}\star
  \hat{\mathscr{F}_{i}}$ is the pushforward of a resolution of
  singularities which is not necessarily small, and is thus of the
  form $Q\oplus Q'$ where $Q'$ is a sum of shifts
  of semi-simple D-modules supported on the diagonal.  

Note that  $\hat{X}^\la_{\xi}$ is the quotient of an affine bundle
over a Grassmannian by a connected algebraic group, and thus simply
connected.  Therefore the pullback of $\hat{\mathscr{E}_{i}}\star
  \hat{\mathscr{F}_{i}}$ to the diagonal is the pushforward by a proper algebraic fiber
  bundle to a simply connected space, and thus a sum of shifts of the
 structure sheaf.  When we use Kashiwara's theorem to think of this as
 a D-module on $\hat{X}^\la_{\xi}\times \hat{X}^\la_{\xi}$, we obtain
 a sum of shifts of $\mathcal{D}_{ \Delta}$.  Furthermore, since the fiber
  is $\mathbb{P}^{\langle\al_i,\xi\rangle+v_i-1}$, we know that it is
  the sum $(q^{\langle\al_i,\xi\rangle+v_i-1}+\cdots +
  q^{-\langle\al_i,\xi\rangle+1})\cdot \mathcal{D}_{ \Delta}$.  On the
  other hand, the pullback of $Q$ is $(q^{\langle\al_i,\xi\rangle+v_i-1}+\cdots +
  q^{\langle\al_i,\xi\rangle+1})\cdot \mathcal{D}_{ \Delta}$ by the
  same argument. This is only possible
  if \[Q'\cong (q^{-\langle\al_i,\xi\rangle
    +1}+\cdots +q^{\langle\al_i,\xi\rangle
    -1}) \cdot \mathcal{D}_{ \Delta}.\]
This shows that (\ref{eq:5}) holds.
If we instead $\langle\al_i,\xi\rangle\leq 0$, we can show (\ref{eq:4})
by applying the same
argument, switching the roles of the two sheaves.

Thus, applying \eqref{ii-red} to  the equations
(\ref{eq:4}--\ref{eq:5}), 
we arrive at the desired isomorphisms (\ref{item:1}--\ref{item:2}).
Thus, by \cite[4.13]{RouQH}, we have a 2-functor from $\tU$
to $\EuScript{Q}^\la$ as desired.
\end{proof}
This completes the proof of Theorem  \ref{th:main}.

\bibliography{../gen}
\bibliographystyle{amsalpha}
\end{document}